\documentclass[12pt]{amsart}
\usepackage[margin=1.0in]{geometry}                
\usepackage{amssymb}
\usepackage{amsmath, amsthm, enumerate}
\usepackage{amsfonts}
\usepackage{mathrsfs}
\usepackage{enumerate}
\usepackage{graphicx}

\newtheorem{thm}{Theorem}[section]
\newtheorem{prop}[thm]{Proposition}
\newtheorem{lemma}[thm]{Lemma}

\numberwithin{equation}{section} 

\theoremstyle{definition}

\newtheorem{defn}[thm]{Definition}

\theoremstyle{remark}

\newcommand{\RR}{\mathbb{R}}
\newcommand{\NN}{\mathbb{N}}

\newcommand{\supp}{\mathrm{supp}\,}

\title{On the maximal directional Hilbert transform} 
\author{ Izabella {\L}aba, Alessandro Marinelli, Malabika Pramanik}
\date{\today}

\begin{document}

\begin{abstract}
For any dimension $n \geq 2$, we consider the maximal directional Hilbert transform $\mathscr{H}_U$ on $\mathbb R^n$ associated with a direction set $U \subseteq \mathbb S^{n-1}$: 
\[ \mathscr{H}_Uf(x) := \frac{1}{\pi} \sup_{v \in U} \Bigl| \text{p.v.} \int f(x - tv) \, \frac{dt}{t}\Bigr|.\] 
The main result in this article asserts that for any exponent $p \in (1, \infty)$, there exists a positive constant $C_{p,n}$ such that for any finite direction set $U \subseteq \mathbb S^{n-1}$,  \[||\mathscr{H}_U||_{p \rightarrow p} \geq C_{p,n} \sqrt{\log \#U}, \] where $\#U$ denotes the cardinality of $U$. As a consequence, the maximal directional Hilbert transform associated with an infinite set of directions cannot be bounded on $L^p(\mathbb{R}^{n})$ for any $n\geq 2$ and any $p \in (1, \infty)$. This completes a result of Karagulyan \cite{Karagulyan}, who proved a similar statement for $n=2$ and $p=2$.
\end{abstract}

\maketitle
\tableofcontents

\section{Introduction}

{\allowdisplaybreaks 
\noindent The fundamental and ubiquitous nature of the classical one-dimensional Hilbert transform has inspired the study of a large variety of operators that share some of its distinctive features. Among the numerous higher-dimensional variants of this transform that are available in the literature, the maximal directional Hilbert transform is of notable interest, in view of its connections with several central problems in harmonic analysis, such as Carleson's theorem on the convergence of Fourier series, estimates on maximal functions of Kakeya type and Stein's conjecture on the Hilbert transform along Lipschitz vector fields.  The treatises \cite{{Lacey-Li-TAMS}, {Lacey-Li-Memoirs}} of Lacey and Li contain an extensive survey of these connections. 
\vskip0.1in 
\noindent Given a unit vector $v \in \mathbb S^{n-1}$, the {\em{directional Hilbert transform}} $\mathscr{H}_v$ is defined initially on Schwartz functions on $\mathbb R^n$ as follows,  
\begin{eqnarray}
\mathscr{H}_{v} f (x):= \frac{1}{\pi} \text{p.v.} \int f(x - tv) \frac{dt}{t} = \frac{1}{\pi} \lim_{\epsilon \rightarrow 0^{+}} \int_{\left|t\right| > \epsilon}   f(x - tv)\frac{dt}{t}, \quad x \in \mathbb R^n. 
\end{eqnarray}
After a rotation that sends $v$ to the first canonical basis vector $e_1 = (1, 0, \ldots, 0)$, this is essentially a tensor product of the classical Hilbert transform in $x_1$ with the identity operator in the remaining variables. As a result, Lebesgue mapping properties of $\mathscr{H}_v$ are easy consequences of its one-dimensional counterpart \cite{{G1}, {G2}, {M-S}}; namely, $\mathscr{H}_v$ is bounded from $L^p(\mathbb R^n)$ to $L^q(\mathbb R^n)$ if and only if $1 < p = q < \infty$.  

\vskip0.1in 
\noindent
The maximal version of the operator $\mathscr{H}_v$, termed the {\em{maximal directional Hilbert transform}}, is the primary object of study in this article. Given a set of unit vectors $U \subseteq \mathbb S^{n-1}$ and initially for a Schwartz function $f$, it is defined to be 
\begin{eqnarray}\label{HT-max}
\mathscr{H}_{U} f(x) := \max_{v \in U} \left|\mathscr{H}_{v} f(x)\right|, \qquad x \in \mathbb R^n.  
\end{eqnarray}
For finite sets $U$, the triangle inequality gives $\mathscr{H}_{U}f \leq \sum_{v \in U} \mathscr{H}_v f$. Thus $\mathscr{H}_U$ continues to be bounded on the same Lebesgue spaces as the classical Hilbert transform, with the trivial bound 
\begin{equation} \label{trivial-bound} ||\mathscr{H}_U||_{p \rightarrow p} \leq \#U || \mathscr{H}_{e_1}||_{p \rightarrow p}, \qquad p \in (1, \infty). \end{equation} 
Here and throughout the paper, $||T||_{p \rightarrow p}$ will denote the operator norm of $T$ from $L^p(\mathbb R^n)$ to itself. This gives rise to the following natural questions: 
\vskip0.1in 
\begin{enumerate}[1.]
\item {\em{To what extent can one improve upon the trivial estimate}} \eqref{trivial-bound}?
\vskip0.1in
\item {\em{Do there exist infinite sets $U$ for which $||\mathscr{H}_U||_{p \rightarrow p}$ is finite for some $p \in (1, \infty)$}}?  
\end{enumerate}   
\vskip0.1in
\noindent For $n=2$, various aspects of question 1 above have been addressed in a large body of work \cite{{Lacey-Li-TAMS}, {Demeter}, {D-D}, {DP-P}}, encompassing results of two distinct types. With $U = \mathbb S^1$ and for $\mathscr{H}_U$ localized to a single frequency scale, Lacey and Li \cite{Lacey-Li-TAMS} have shown that the operator $f \mapsto \mathscr{H}_{\mathbb S^1} (\zeta \ast f)$ maps $L^2$ into weak $L^2$, and $L^p$ to itself for $p > 2$. Here $\zeta$ is a Schwartz function with frequency support $\{1 \leq |\xi| \leq 2\}$. For finite $U$ and in the unrestricted setting (i.e., without any Fourier localization), $\mathscr{H}_U$ has been studied in the more general context of maximal directional singular integral operators, co-authored in part by Demeter, Di Plinio and Parissis. For instance, the main results in \cite{{Demeter}, {D-D}} give that for a general direction set $U \subseteq \mathbb S^1$, 
\[ ||\mathscr{H}_U||_{p \rightarrow p} \leq C_p \log \#U, \qquad  2 \leq p < \infty, \]
where $C_p$ is a constant independent of $U$. For $p = 2$, this upper bound is in fact sharp for the uniformly distributed set of directions \[ U = \{e^{\frac{2 \pi i k}{N}} : k = 1, \ldots, N \}, \] 
see \cite[Section 3]{Demeter}. On the other hand, for lacunary sets $U \subseteq \mathbb S^1$ of finite order defined as in \cite{{D-D}, {DP-P}}, it has been shown that 
\begin{equation}  ||\mathscr{H}_U||_{p \rightarrow p} \leq C_p \sqrt{\log \#U}, \qquad 1 < p < \infty,  \label{lacunarity-upper} \end{equation}     
where the constant $C_p$ also depends on the lacunarity order of $U$.  For $n \geq 3$, partial results with $p = 2$ are due to Kim \cite[Theorem 2]{Kim}. Specifically, the estimate \[ ||\mathscr{H}_U||_{2 \rightarrow 2} \leq C N^{\frac{n-2}{2}}\] 
is shown to hold for a direction set $U \subseteq \mathbb S^{n-1}$ of cardinality $N^{n-1}$ in general position contained inside the positive orthant. The bound is shown to be sharp for a member of this class. 
\vskip0.1in
\noindent In contrast, question 2 is much less studied in complete generality. Even though phrased in terms of infinite direction sets, after a finitary and quantitative reformulation it is really a question about lower bounds on $||\mathscr{H}_U||_{p \rightarrow p}$ for general $U$. A result of Karagulyan \cite{Karagulyan} addresses this question in the planar setting and for $p = 2$, obtaining a lower bound of order $\sqrt{\log \#U}$ for $||\mathscr{H}_U||_{2 \rightarrow 2}$ in this case. The goal of this paper is to establish this bound in far greater generality, extending it to all exponents $p \in (1, \infty)$ and to all dimensions $n \geq 2$.  
For convenience, all logarithms below will be taken to the base 2.    
\begin{thm}\label{our Theorem2}
Let $U$ be a finite set of unit vectors in $\mathbb{R}^{n}$ with $n \geq 2$. Then for $1 < p < \infty$, there exists a positive constant $C_{p,n}$ such that 
\begin{eqnarray}\label{fondestim} 
\left\|\mathscr{H}_{U}\right\|_{p \rightarrow p}  \,\,  \geq  \,\,  C_{p,n}  \sqrt{ \log \#U}, 
\end{eqnarray}
where $\#U$ is the cardinality of the set $U$.
\end{thm}
\noindent {\em{Remarks: }} \begin{enumerate}[1.] 
\item Since the single-vector Hilbert transform $\mathscr{H}_v$ is not bounded as an operator on $L^1(\mathbb R^n)$ or on $L^{\infty}(\mathbb R^n)$ or from $L^p(\mathbb R^n)$ to $L^q(\mathbb R^n)$ for $p \ne q$, the theorem is trivially true for these exponents.   
\vskip0.1in 
\item The lower bound in \eqref{fondestim} is in fact attained by certain direction sets $U$, as \eqref{lacunarity-upper} shows. This gives rise to an interesting question: which geometric properties of a direction set $U$ dictate the growth rate of $||\mathscr{H}_U||_{p \rightarrow p}$? 
\vskip0.1in

\item Our result extends easily to the periodic setting, with a similar proof. More explicitly, if $\mathscr{H}_U$ is viewed as an operator from $L^p(\mathbb T^n)$ to $L^q(\mathbb T^n)$, where $\mathbb T^n$ denotes the $n$-dimensional unit torus, then our arguments show that \[ ||\mathscr{H}_U||_{L^{p}(\mathbb T^n) \rightarrow \mathbb L^q(\mathbb T^n)} \geq C_{p,q,n} \sqrt{\log \#U}, \]
for all $p, q \in (1, \infty)$ with $q \leq p$. The operator is unbounded for all other choices of $p,q$. 
The construction of test functions on the torus proceeds similarly, except the convolution in 
(\ref{def-fN}) is taken on $\mathbb{T}^n$ instead of $\mathbb{R}^n$.

\vskip0.1in
\item As a consequence of \eqref{fondestim}, we are able to conclude the unboundedness of $\mathscr{H}_U$ for {\em{all}} infinite direction sets $U$ in all dimensions and on all nontrivial Lebesgue spaces. We record this below.  
\end{enumerate} 
\begin{thm}\label{our Theorem1}
For any infinite set of unit vectors $U$ in $\RR^{n}$ with $n \geq 2$, the operator $\mathscr{H}_{U}$ cannot be extended to a bounded operator on $L^{p}( \mathbb{R}^{n} )$ for any $1 < p < \infty$.
\end{thm}
\noindent This is in sharp contrast with the behaviour of the closely related maximal directional operator 
\[ \mathscr{M}_Uf(x) :=  \sup_{v \in U} \sup_{r > 0} \frac{1}{2r} \int_{-r}^{r} |f(x - tv)|\, dt, \] 
whose Lebesgue boundedness is {\em{not}} connected with the finitude of $U$. For instance, the operator $\mathscr{M}_U$ is known to be $L^p$-bounded for all $p \in (1, \infty]$ if $U$ is an infinite direction set of lacunary type in $\mathbb R^n$, see for example \cite{{A}, {Carb}, {N-S-W}, {P-R1}, {P-R2}, {S-S}}. For other types of direction sets that lack the feature of finite-type lacunarity, the operator $\mathscr{M}_U$ is known to be unbounded on $L^p$ for all $p \in [1, \infty)$. This has been studied in \cite{{Bateman}, {Bateman-Katz}, {Katz}, {Kroc-P}}. 

\subsection{Notation and a preliminary reduction} We recall the equivalent Fourier-analytic formulation of the problem. For functions
$f, g \in L^2(\RR^n)$, we will use $\widehat{f}$ and $g^{\vee}$ to denote the Fourier transform and inverse Fourier transform respectively, 
$$
\widehat{f}(\xi):=\int f(x) e(-\xi\cdot x) dx, \qquad g^\vee(x):=\int g(\xi) e(\xi\cdot x) d\xi,
$$
where $e(t) :=e^{2\pi i t}$ for $t\in\RR$. If $E\subset \RR^n$ is a measurable set, we will use $\chi_E$ to denote its characteristic function, and $|E|$ to denote its Lebesgue measure. Given a unit vector $v\in \RR^{n}$, we will use $\Gamma_v$ to denote the half-space 
\begin{eqnarray} \label{half-space}
\Gamma_{v}  := \left\{x \in \mathbb{R}^{n} : x \cdot v > 0\right\}.
\end{eqnarray}
\vskip0.1in
\noindent It is well known \cite{{G1}, {G2}, {M-S}} that the classical one-dimensional Hilbert transform $H$ can be expressed as a Fourier multiplier operator,
\[Hf(x)  = (-i \,\hbox{sgn\,}(\cdot)\widehat{f} \  )^{\vee} (x).\]
For the directional Hilbert transform, this means that  
$\mathscr{H}_{v} f =  -i \left[ (2\chi_{\Gamma_{ v} } \widehat{f} \ )^{ \vee} - f \right]$. 
Accordingly, we define 
\[
T_{v}f := ( \chi_{\Gamma_{v}} \widehat{f}\,)^{\vee}, \qquad \text{ and } \qquad T_{U} f (x)   :=   \max_{v \in U} |T_{v}f (x)|.
\]
Thus the boundedness of (\ref{HT-max}) is equivalent to that of $T_U$.
In particular,  Theorem \ref{our Theorem2} is equivalent to the bound
\begin{eqnarray}\label{my fond estim changed}
\|T_{U} \|_{p \rightarrow p}   \geq   C_{p,n}  
\sqrt{ \log \# U}, \ \ 1 < p <\infty.
\end{eqnarray}


\subsection{Overview of the proof} The proof of \eqref{my fond estim changed} relies on three main components. One is geometric. More precisely, a suitable pruning and ordering of the direction set $U = \{u_1, \cdots, u_{2^m} \}$ generates a finite number of mutually disjoint conic sectors $S_N \subseteq \mathbb R^n$, with the property that  $S_N$ is contained in $\Gamma_{u_k}$ if $N \leq k$ and is disjoint from $\Gamma_{u_k}$ otherwise. This part of the argument is greatly simplified in the planar setting, but needs a little more care in general dimensions. This geometric ingredient is contained in Lemma \ref{Lemma 9}. Its proof is presented in Section \ref{geometric-section}.  
\vskip0.1in
\noindent The second ingredient is analytical. Following the general guidelines of \cite{Karagulyan} and given a fixed Lebesgue exponent $p$, the sectors $S_N$ are used to construct a test function $f$ of the form $f = \sum f_N$, based on which \eqref{my fond estim changed} will be verified. On one hand, the function $f_N$ is frequency-supported in a large cube contained in the sector $S_N$. Not only are these cubes disjoint from one another, they are strongly separated in a way that ensures a high degree of orthogonality among the various summands $f_N$. On the other hand, the essential spatial support of $f_N$ is in a set $E_N \subseteq [0,1]^{n}$, with the property that any two sets in the collection $\{ E_N \}$ are either disjoint or nested. The critical features of this iterative construction of $f$ have been laid out in Proposition \ref{building blocks proposition} of Section \ref{proof of thm}, and the proof of \eqref{my fond estim changed} appears here, modulo the two main estimates 
\begin{equation}  ||f||_p \lesssim_p \sqrt{m} \quad \text{ and } \quad |\{x \in \mathbb R^n: |T_Uf(x) \gtrsim m \}| \gtrsim 1,  \label{main-estimates-for-proof}  \end{equation}  
the details of which are given in later sections. 
\vskip0.1in
\noindent The proof of the estimates in \eqref{main-estimates-for-proof} constitutes the combinatorial component of the argument. Section \ref{norm-estimate-section}  contains the steps that lead to the first inequality in \eqref{main-estimates-for-proof}. The nested structure of the sets $E_N$ is best encoded as a binary tree. Combined with the stringent frequency localizations imposed on $f_N$, this results in an upper bound on $||f||_p$ that is essentially comparable to $||f||_2$. Choosing $p$ a large even integer without loss of generality allows us to express $||f||_p^p$ as the sum of a large number of terms of the form \[ \int h_{N_1} \cdots h_{N_p} \quad \text{ where }  h_N \text{ is either } f_N \text{ or } \overline{f}_N. \]  Many of these terms can be ignored, based on disjoint spatial and frequency support considerations. The language of trees aids greatly in the book-keeping, identifying strings of indices $(N_1, \cdots, N_p)$ that genuinely contribute to the norm. This segment of the proof has no corresponding counterpart in \cite{Karagulyan}, where $p$ was always 2. 
\vskip0.1in
\noindent In addition, the choice of modulation parameters in $f_N$ endows the functions Re$(f_N) = \varphi_N$ with Haar function-like properties, termed ``signed tree systems". Basic materials concerning trees and signed tree systems have been collected in Section \ref{tree-section}. An important fact concerning a signed tree system $\{\varphi_N\}$ is proved in Section \ref{tree-section-permutation}: namely, for a given $m$ and despite obvious oscillations, there exists a universal permutation $\sigma$ of $\{1, \cdots, 2^m-1 \}$ for which the largest partial sum 
\[ \max_{1 \leq l \leq 2^m} \Bigl| \sum_{N=1}^{l} \varphi_{\sigma(N)} \Bigr| \text{ is comparable to } \sum_{N=1}^{2^m-1} |\varphi_N|. \]
This choice of $\sigma$ dictates the ordering of the sectors $S_N$ and is critical to the second estimate in \eqref{main-estimates-for-proof}.
   
\section{Trees and tree systems} \label{tree-section} 

\subsection{Trees} \label{just-trees-subsection} Given a large positive integer $m$, we will use the following system of double-indexing to keep track of a large collection of sets and functions that arise in the sequel. Any positive integer $1 \leq N \leq 2^m$ will be identified with the pair $(k,j)$, where 
\begin{eqnarray}\label{e-jk}
N = 2^{k} + j - 1,  \quad  1 \leq j \leq 2^{k},  \quad k = 0, 1, \ldots  m-1.
\end{eqnarray}
As indicated in the introduction, the language of binary trees is a convenient tool in depicting this double-indexing system.  Consider a full binary tree $\mathcal T_m$ of height $m$, and label each tree vertex as $(k,j)$, where $k$ is the \emph{height} of the vertex (so that $k$ ranges from $0$ to $m-1$), and all vertices of height $k$ are labelled lexicographically as $(k,1), (k,2), \dots,(k,2^k)$. Given a vertex $(k,j)$, 
\begin{itemize} 
\item its {\em{parent}} can be identified as $\left(k-1, \lfloor \frac{j+1}{2} \rfloor \right)$ if $k \geq 1$, and 
\item its {\em{left and right children}} can be identified respectively as $(k+1,2j-1)$ and $(k+1,2j)$ if $k \leq m-2$. 
\end{itemize} 
A \emph{ray} $\mathscr{R}$ of length $l+1$ rooted at $(k, j_{0})$ is a sequence of vertices $\left\{(k, j_{0}), (k+1, j_{1}), \ldots (k+l, j_{l})\right\}$ where, for each $i = 1, \ldots l$, the vertex $(k+i, j_{i})$ is a child of $(k+i-1, j_{i-1})$.
We will also say that the vertex $(k',j')$ is  a \emph{descendant} of $(k,j)$ if $k'>k$ and 
$(k,j)$ lies on the ray connecting $(k',j')$ to the root $(0,1)$ of the tree.
\vskip0.1in
\noindent In parallel to the double-numbering system in (\ref{e-jk}), we will use a similar convention for tree vertices, so that the vertex $(k,j)$ will be alternatively labelled by the number $N(k,j) = 2^{k} + j -1$. We will use $h(N) = k$ to denote the height of the vertex  $N$.

\subsection{Tree systems} \label{tree-systems-subsection} Let $Q=[0,1]^n$. We will consider finite sequences of functions $\{f_{N} : 1 \leq N \leq 2^m-1 \}$, where each $f_N$ is a complex-valued function supported in $Q$, and use our double-numbering system from Section \ref{just-trees-subsection} to order the sequence. Thus
\begin{align*}
&f_{1} =  f^{(0)}_{1}, 
\\
&f_{2}  = f^{(1)}_{1},  \ \   f_{3}  =  f^{(1)}_{2},
\\
&f_{4}  =  f^{(2)}_{1},  \ \   f_{5}  = f^{(2)}_{2},  \ \   f_{6} =  f^{(2)}_{3}, \ \   f_{7} =  f^{(2)}_{4}, \text{etc}. 
\end{align*}
In many of our applications, the sequences of functions $f_N$ will satisfy at least one of the following properties: 
\begin{itemize} 
\item For any pair $f_N,f_M$ with $N\neq M$, the supports of $f_N$ and $f_M$ are either nested or disjoint, or 
\item More specifically, if $M>N$, then $f_N$ has constant sign on the support of $f_M$ (up to sets of Lebesgue measure 0). 
\end{itemize} 
A prototype of such a system is provided by Haar functions (cf. \cite{N-U}). The abstract formulation of the property we need was given by Karagulyan in \cite{Karagulyan}. We follow the rough outline of Karagulyan's presentation in the definitions below, but also modify the terminology and use the language of graph theory more extensively in order to accommodate the later parts of the proof that are not present in \cite{Karagulyan}. In particular, we use the term ``signed tree systems" to refer to the ``tree systems" of \cite{Karagulyan}.

\begin{defn}\label{tree-def}
Let $f_N = f_j^{(k)}$, $1 \leq N \leq 2^m-1$, be a finite sequence of functions, indexed as above with $N, j, k$ related by \eqref{e-jk}. 
\begin{enumerate}[(a)] 
\item \label{tree-system-def} We say that $\{f_{N}\}$ is a \emph{tree system} if
 the following holds Lebesgue almost everywhere on $Q$:
\begin{equation}\label{unsigned-tree}
\supp f^{(k+1)}_{2j-1} \cap \supp f^{(k+1)}_{2j} =\emptyset, \qquad \supp f^{(k+1)}_{2j-1} \cup \supp f^{(k+1)}_{2j} \subset \supp  f^{(k)}_{j}.
\end{equation}
\vskip0.1in 
\item \label{signed-tree-system-def} We say that the sequence $\{f_{N}\}$ is a \emph{signed tree system} if
the following holds Lebesgue almost everywhere on $Q$:
\begin{align}\label{tree system condition 1}
\supp f^{(k+1)}_{2j-1}   &\subset   \bigl\{x \in Q : f^{(k)}_{j}(x) > 0\bigr\}, \\ 
\label{tree system condition 2}
\supp f^{(k+1)}_{2j}   &\subset  \bigl\{x \in Q : f^{(k)}_{j}(x) < 0\bigr\}.
\end{align}
Note that if $N=2^k + j -1$, then $2^{k-1}+\lfloor \frac{j+1}{2} \rfloor -1 =  \lfloor \frac{N}{2}\rfloor =: N^{\ast}$. In particular, \eqref{tree system condition 1} and \eqref{tree system condition 2} are equivalent to  
\begin{equation}\label{tree system condition 3}
\supp f_{N}   \,\,  \subset  \,\,    \Big\{ x \in Q : (-1)^{j-1}  f_{N^{\ast}}(x) > 0 \Big\}.
\end{equation}
\end{enumerate} 
\end{defn}
\noindent Figure \ref{fig1} shows the relations among the supports. Clearly, every signed tree system is a tree system. 
\begin{figure}[htbp]\centering\includegraphics[width=4.in]{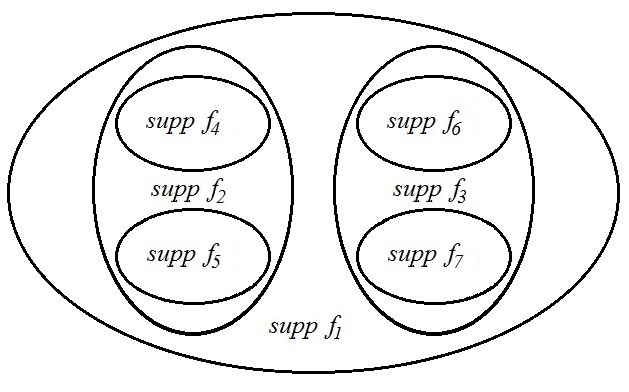}
\caption{The nested supports of tree system functions.}
\label{fig1}
\end{figure}

\vskip0.1in
\noindent The terminology of trees adapts easily to tree systems of functions. Thus for $k=0,1,\dots, m-2$, each function $f_j^{(k)}$ in a tree system is identified with a vertex in a complete binary tree of height $m$, and has two children $f^{(k+1)}_{2j-1}$ and $f^{(k+1)}_{2j}$ with mutually disjoint (up to sets of measure 0) supports, both supported on $\supp f_j^{(k)}$. In a signed tree system, we have
the additional property that 
the left child
$f^{(k+1)}_{2j-1}$ of $f^{(k)}_j$ is supported on the set where $f^{(k)}_{j}(x) > 0$, and the right child
$f^{(k+1)}_{2j}$ is supported on the set where $f^{(k)}_{j}(x) <0$. Iterating this, we get the following.
\begin{lemma}\label{lemma-nesting}
Let $\{f_{N}: N = 1, \ldots, 2^m-1\}$ be a tree system. Then
for $(k,j)\neq (k',j')$, the supports of $f_j^{(k)}$ and $f_{j'}^{(k')}$ are either disjoint or nested. Moreover:
\vskip0.1in 
\begin{enumerate}[(a)] 
\item If
$(k,j)$ and $(k',j')$ do not lie on the same tree ray (i.e. neither vertex is a descendant of the other), 
then the supports are disjoint. 
\vskip0.1in
\item If $(k',j')$ is a descendant of $(k,j)$, then $\supp f_{j'}^{(k')} \subset \supp f_{j}^{(k)}$.
\vskip0.1in
\item \label{ray-part} For each $x\in Q$ such that $f_1^{(0)}(x)\neq 0$, there is a unique maximal ray \[ \mathscr{R}(x)
=\{ (0,1), (1,j_1),\dots, (k,j_k) \} \text{ such that } f_{j_i}^{(i)}(x)\neq 0  \text{ for }  i=1,\dots,k. \] 
The ray terminates at $ f_{j_k}^{(k)}$ when either $k=m-1$ or both children of $ f_{j_k}^{(k)}$ take value $0$ at $x$.
\vskip0.1in
\item \label{left-right-part} If $\{f_{N} : 1 \leq N \leq 2^m-1\}$ is a signed tree system, then we have the additional property
that $\mathscr{R}(x)$ encodes the sign of $f_{1}^{(0)}, f_{j_1}^{(1)},\dots, f_{j_{k-1}}^{(k-1)}$
at $x$: 
\[ \begin{cases} &\text{it turns left at $(i,j_i)$
(i.e.\ goes to the left child $(i+1,2j_i-1)$)  if $f^{(i)}_{j_i}(x) > 0$, and } \\ \\  &\text{it turns right at $(i, j_i)$ (i.e.\ goes to the right child $(i+1, 2j_i)$) if $f^{(i)}_{j_i}(x) < 0$.} \end{cases} \] 
\end{enumerate} 

\end{lemma}

\begin{figure}[htbp]\centering\includegraphics[width=4.in]{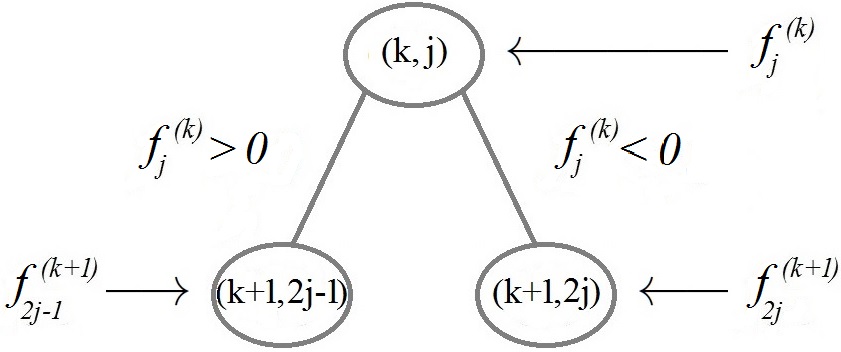}
\caption{For a signed tree system, the signs of $f_j^{(k)}$ are encoded in the binary tree.}
\label{fig2}
\end{figure}
%


\subsection{Choice of the permutation $\sigma$} \label{tree-section-permutation}
Here we define a special permutation $\sigma$ of $\{1,\dots,2^m-1\}$ that plays a central role in the subsequent analysis.
For $N = 2^{k} + j -1$ as in (\ref{e-jk}), let
\begin{eqnarray} \label{def-tN}
t_{N}   =   t^{(k)}_{j}    :=  \frac{2j-1}{2^{k+1}}   \in  \left[0,1\right].
\end{eqnarray}
Thus $t_1=\frac{1}{2},\ t_2=\frac{1}{4}, \ t_3=\frac{3}{4},\ t_4=\frac{1}{8},\ t_5=\frac{3}{8}, \ t_6=\frac{5}{8}$, etc. 
Observe that a complete binary tree with $m$ levels can be represented as a planar graph
so that the vertex $(k,j)$ has the $x$-coordinate $t_j^{(k)}$. Figure \ref{fig3}
illustrates this for $m=3$.

\begin{figure}[htbp]\centering\includegraphics[width=4.in]{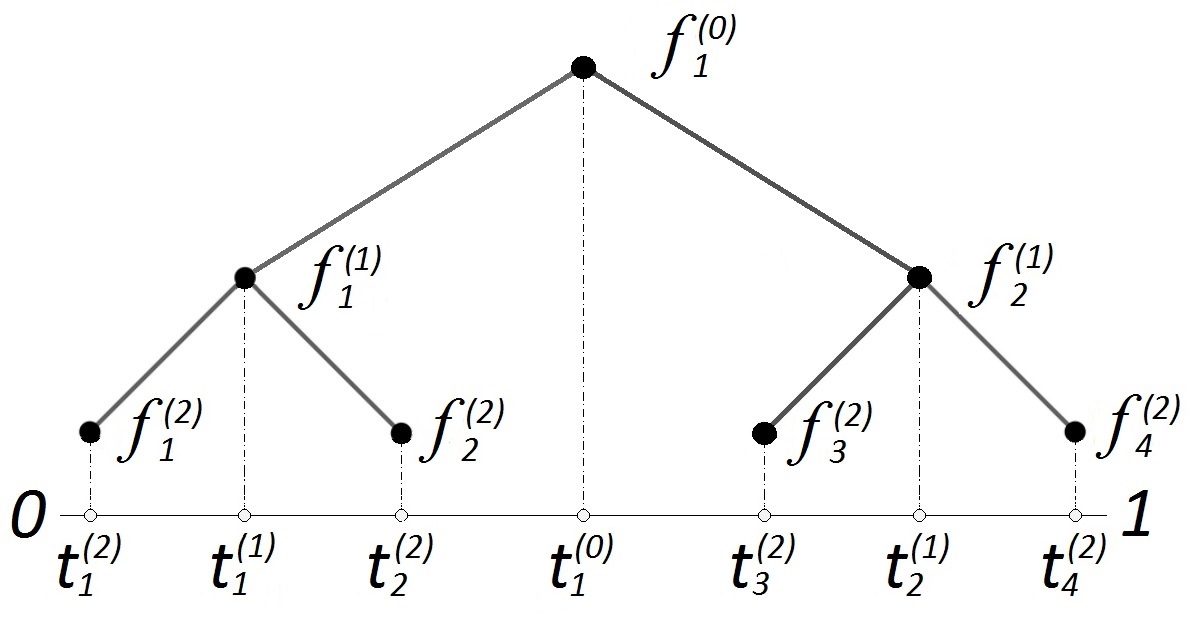}
\caption{The full binary tree with $m=3$ and with the numbers $t_N$.}
\label{fig3}
\end{figure}
\vskip0.1in
\noindent We now rearrange the sequence $\{t_{N}\}\subset[0,1]$ in increasing order. Specifically, there exists a unique permutation $\sigma$ of the numbers $\left\{1, 2, \ldots 2^{m}-1\right\}$, depending only on $m$, such that 
\begin{eqnarray} \label{definition of sigma}
t_{\sigma(1)}    <  t_{\sigma(2)}   <    \ldots   < t_{\sigma(2^{m}-1)}.
\end{eqnarray}
(In the example in Figure 3, we have $\sigma(1)=4,\sigma(2)=2, \sigma(3)=5$, etc.) 
\vskip0.1in
\noindent If $f_1,\dots,f_{2^m-1}$ are Haar functions on the line, then for each $N=1,\dots,2^m-1$
the number $t_N$ is the coordinate of the point where $f_N$ changes sign from positive to negative, and the permutation
$\sigma$ arranges the sequence $\{t_N\}$ in increasing order. This observation leads directly to a special case, due to Nikishin and Ulyanov \cite{N-U}, of  Lemma \ref{Karagulyan's tree systems lemma} below. The generalization of the lemma to general tree systems is due to Karagulyan \cite{Karagulyan}; while Karagulyan states it only for $n=2$, the same proof works in all dimensions. We follow the argument of \cite{Karagulyan}, with minor corrections\footnote{Karagulyan uses $f_{\sigma(h)}(x) \leq 0$
instead of $f_{\sigma(h)}(x) < 0$ in his definition of $l_x$. With that definition, the property
(\ref{tree system claim}) does not necessarily hold. We have rewritten that part of the proof. 
Alternatively, Karagulyan's proof does work if the definition of $l_x$ is changed to 
$l_{x} = \max  \left\{ h : 1 \leq h < 2^{m},\  f_{\sigma(k)}(x) \leq  0\  \forall\  1\leq k\leq h \right\}$ instead.
We thank an anonymous referee, as well as G.~Karagulyan (private communication), for bringing that to our attention.
 }
and expository changes.


\begin{lemma} (\cite[Lemma 1]{Karagulyan}) \label{Karagulyan's tree systems lemma} 
If $\sigma$ is the permutation defined in \eqref{definition of sigma}, then for every signed tree system $f_1, \dots , f_{2^{m}-1}$ in $\mathbb{R}^{n}$ we have
\begin{eqnarray}\label{tree system inequality}
\max_{1 \leq l < 2^{m}}  \Bigl| \sum^{l}_{N = 1} f_{\sigma(N)}(x) \Bigr|  \,\,  \geq  \,\,  \frac{1}{3}  \sum^{2^{m}-1}_{N = 1}   \left|f_{N}(x)\right|.
\end{eqnarray}
for all $x \in \mathbb{R}^{n}$.
\end{lemma}

\begin{proof}
In view of \eqref{def-tN}, we find that  
\begin{equation}\label{e-tdescent}
t^{(k+1)}_{2j-1}=t_j^{(k)}-\frac{1}{2^{k+2}},\ \ 
t^{(k+1)}_{2j}=t_j^{(k)}+\frac{1}{2^{k+2}}.
\end{equation}
Iterating over tree levels from $k+1$ to $m$, and using that $\sum_{i=k+2}^{k'} 2^{-i}< \sum_{i=k+2}^\infty 2^{-i}
= 2^{-k-1}$ for any finite $k'\geq k+2$, we see that whenever $N' = N'(k', j')$ is a descendant 
of $N = N(k,j)$ in the binary tree, the corresponding numbers 
$t_{j'}^{(k')}$ obey
\begin{equation}\label{e-tbound}
t_{j'}^{(k')}\in \left(t_j^{(k)}-2^{-k-1}, t_j^{(k)}+2^{-k-1} \right).
\end{equation}
\vskip0.1in 

\noindent Let $x\in Q$, and assume that $f_1(x)\neq 0$ since otherwise there is nothing to prove. Define
\begin{eqnarray}\label{def-lx}
l_{x} := \max  \left\{ h : 1 \leq h < 2^{m},\  f_{\sigma(h)}(x) < 0\right\},
\end{eqnarray}
with the convention that $l_x=0$ if the set above is empty. It follow immediately from (\ref{def-lx}) that $f_{\sigma(h)}(x)\geq 0$ for all $h  >  l_{x}$. 
We claim that, furthermore,
\begin{eqnarray} \label{tree system claim}
f_{\sigma(h)}(x)  \leq  0 \ \  \forall  \,\,  h \leq l_{x}.
\end{eqnarray}
To prove this, suppose 
for contradiction that there exists an $h\leq l_x$ such that $f_{\sigma(h)}(x) >0$. By (\ref{def-lx}), we cannot have
$h=l_x$. Let 
$$
\sigma(l_{x})  = N(k, j), \qquad   \sigma(h) = N(s, i).
$$
Consider the ray $\mathscr{R}(x)$ defined in Lemma \ref{lemma-nesting} \eqref{ray-part}. By definition, $\mathscr{R}(x)$ contains the vertex $\sigma(l_x)$.  
If $\sigma(h)\notin \mathscr{R}(x)$, then $f_{\sigma(h)}(x) =0$ and the claim is true. Thus we are reduced to the case where $\sigma(h)$ and $\sigma(l_x)$ both lie on $\mathscr{R}(x)$. 
\vskip0.1in
\noindent Suppose that ${\sigma(h)}$ is a descendant of ${\sigma(l_x)}$. Since $f_{\sigma(l_x)}(x)<0$, Lemma \ref{lemma-nesting} \eqref{left-right-part} 
dictates that the ray $\mathscr{R}(x)$ turns right at $\sigma(l_x)$, so that ${\sigma(h)}$ must be either 
$N(k+1, 2j)$ or one of its descendants. By (\ref{e-tbound}) and then (\ref{e-tdescent}), we have 
$$
t_{\sigma(h)} > t^{(k+1)}_{2j}-\frac{1}{2^{k+2}}= t_j^{(k)}= t_{\sigma(l_{x})} ,
$$
which contradicts the assumption that $h< l_x$ and therefore $t_{\sigma(h)}< t_{\sigma(l_{x})}$.
\vskip0.1in
\noindent Finally, consider the case when ${\sigma(l_x)}$ is a descendant of ${\sigma(h)}$. If $f_{\sigma(h)}(x) >0$,
then $\mathscr{R}(x)$ turns left at ${\sigma(h)}$, so that ${\sigma(l_x)}$ must be either
$N(s+1, 2i-1)$ or one of its descendants. Then, again by (\ref{e-tbound}) and then (\ref{e-tdescent}), we have 
$$
t_{\sigma(l_x)} < t^{(s+1)}_{2i-1}+\frac{1}{2^{s+2}}= t_i^{(s)}= t_{\sigma(h)} ,
$$
again contradicting our assumptions.

\begin{figure}[htbp]\centering\includegraphics[width=4.in]{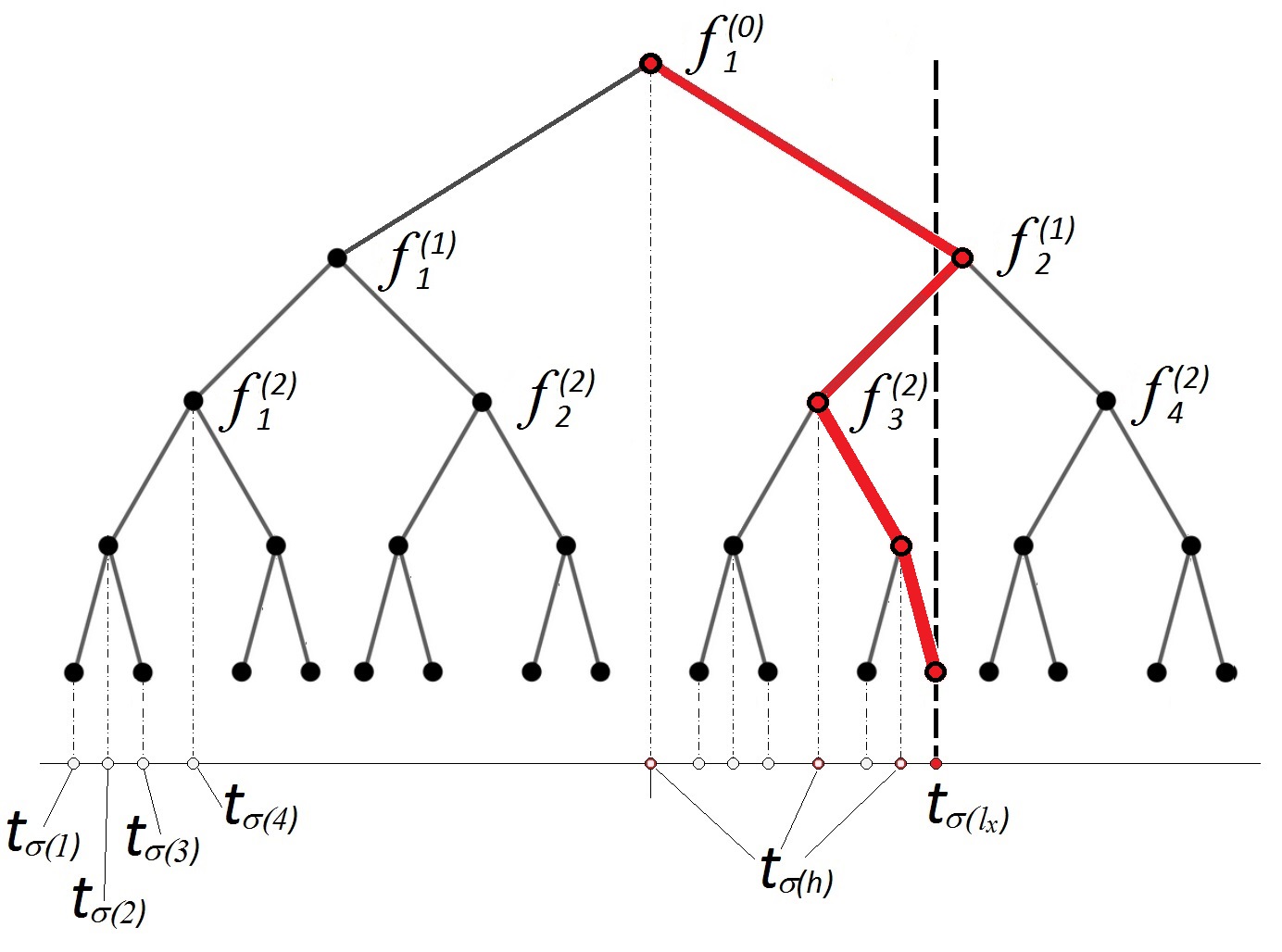}
\caption{Choice of $l_x$ for the example in Section \ref{l-x-example-section} in the case $f^{(4)}_{12}(x) < 0$.
The vertices in $\mathscr{R}^{\ast}(x)$ lie on the dashed line and to the left of it.}
\label{fig4}
\end{figure}

\vskip0.1in
\noindent To recap, we have established the existence of an integer $l_x\geq 0$ such that
\begin{eqnarray} \label{l-x-again}
\left\{ \begin{array}{ll}
f_{\sigma(h)}(x)  \leq  0  & \forall \,\, h \leq l_{x}\\
f_{\sigma(h)}(x)  \geq  0  & \forall \,\, h > l_{x}.\\
\end{array} \right.
\end{eqnarray}
Let
$$
S_1=\sum_{N=1}^{l_{x}} -f_{\sigma(N)}(x),\ \ 
S_2=\sum_{N=l_{x}+1}^{2^m-1} f_{\sigma(N)}(x),
$$
with the convention that $S_1=0$ if $l_x=0$.
Then $S_1,S_2\geq 0$,
\begin{align*}
\sum^{2^{m}-1}_{N = 1}   \left|f_{N}(x)\right| &=S_1+S_2, \text{ and } \\
\max_{1 \leq l < 2^{m}}  \Bigl| \sum^{l}_{N = 1} f_{\sigma(N)}(x) \Bigr|
&= \max (S_1,S_2-S_1).
\end{align*} 
If $S_1\geq \frac{1}{3}(S_1+S_2)$, then \eqref{tree system inequality} follows immediately.
Suppose now that $S_1 <\frac{1}{3}(S_1+S_2)$. Then $S_2> \frac{2}{3}(S_1+S_2)$,
and furthermore, $3S_1<S_1+S_2$ so that $2S_1<S_2$. Hence
$$
S_2-S_1>S_2-\frac{S_2}{2}=\frac{S_2}{2}>\frac{1}{3}(S_1+S_2),
$$
and (\ref{tree system inequality}) again follows.
\end{proof}

\subsubsection{An example} \label{l-x-example-section} The permutation $\sigma$ arranges $t_{\sigma(N)}$ in increasing order. The integer $l_x$ used in Lemma \ref{Karagulyan's tree systems lemma} then has a geometric interpretation in terms of the binary tree $\mathcal T_m$. Given $x \in Q$ and the ray $\mathscr{R}(x)$ as in Lemma \ref{lemma-nesting}(\ref{ray-part}), let $\mathscr{R}^{\ast}(x)$ be the subcollection of vertices on $\mathscr R(x)$ where the ray turns right. The maximal element $(k, j_k)$ is included in $\mathscr{R}^{\ast}(x)$ if and only if $f^{(k)}_{j_k}(x) < 0$. Since the right child (and all its descendants) of any vertex $N$ generate larger $t$-values than $N$ itself, the relation \eqref{def-lx} defining $l_x$ is equivalent to the condition that  $\sigma(l_x) = \max\{N : N \in \mathscr{R}^{\ast}(x)\}$.  
\vskip0.1in
\noindent We explain the choice of $l_x$ in the context of an example given by Figure \ref{fig4}, with $m=5$. Let $x \in Q$ be a point such that \[ \mathscr{R}(x) = \{(0,1), (1,2), (2,3), (3, 6), (4, 12) \}.  \] Then \[ \mathscr{R}^{\ast}(x) = \begin{cases} \{(0,1), (2, 3), (3, 6), (4, 12)  \} &\text{ if }  f_{27}(x) = f^{(4)}_{12}(x) < 0, \\  \{(0,1), (2, 3), (3, 6)  \} &\text{ if }  f_{27}(x) = f^{(4)}_{12}(x) > 0, \end{cases} \]
and hence \[ \sigma(l_x) = \begin{cases} N(4, 12) = 27 &\text{ if } f_{27}(x) < 0, \\ N(3, 6) = 13 &\text{ if }  f_{27}(x) > 0. \end{cases} \]     

\bigskip

\section{Proof of Theorem \ref{our Theorem2}}\label{proof of thm}
\noindent A sector in $\mathbb{R}^{n}$ is an open conic region in Euclidean space bounded by a finite number of hyperplanes passing through the origin. More precisely, 
\begin{defn} \label{Definition of sector}
Let $v_{1}, v_{2}, \ldots v_{r}$ be distinct unit vectors in
$\mathbb{R}^{n}$, and fix an integer $s \leq r$. A \emph{sector} in
$\mathbb{R}^{n}$ is a nonempty set of the form 
\begin{eqnarray}
X = \Big\{x \in \mathbb{R}^{n} : \begin{array}{ll}
x \cdot v_{j} > 0  & \forall \,\, 1 \leq j \leq s\\
x \cdot v_{j} < 0  & \forall \,\, s < j \leq r\\
\end{array}\Big\}
\end{eqnarray}
\end{defn}
\noindent Note that if $x \in X$, then $tx \in X$ for any $t > 0$. Thus a sector
is infinite with nonempty interior, by definition. 
\vskip0.1in
\noindent We record in Section \ref{building-block-section} two results (Proposition \ref{building blocks proposition} and Lemma \ref{Lemma 9}), one analytic and the other geometric, concerning sectors. These results are critical components of the proof of Theorem \ref{our Theorem2}. We present the proof of this main theorem later in Section \ref{main-proof-section}, modulo the two ingredients.  The proofs of the two building blocks appear later in the paper (in Section \ref{building-block-proof-section} for Proposition \ref{building blocks proposition} and Section \ref{geometric-section} for Lemma \ref{Lemma 9}). 
\subsection{The ingredients of the proof} \label{building-block-section}
\begin{prop} \label{building blocks proposition}
For any choice of integer $p_{0} \geq 1$,  there exist constants
$C_{1}, C_{2}, C_{3} > 0$ that depend only on $p_0$ and the ambient
dimension $n$ and satisfy the properties listed below. 
\vskip0.1in
\noindent Let $\{X_{N}: N = 1, 2, \ldots, 2^m-1\}$ be any finite collection of pairwise disjoint
sectors in $\mathbb{R}^{n}$.  Then there exists
a corresponding sequence $\{f_{N} : N=1, \ldots, 2^{m} -  1\}$ of smooth, integrable functions with
compactly supported Fourier transforms such that:
\begin{enumerate}[(a)]
\item \label{support-cone} $\supp \widehat{f}_{N} \subset X_{N}$ for each $N$.
\vskip0.1in
\item  \label{f-pnorm} For each $p \in [1, 2p_0]$, 
\begin{equation} 
    \Big\|  \sum^{2^{m} - 1}_{N = 1} f_{N} \Big\|_p \leq
    C_{1} \sqrt{m}. 
\end{equation}   
\item \label{f-sup-norm} For the permutation $\sigma$ defined in \eqref{definition of sigma} and
used in Lemma \ref{Karagulyan's tree systems lemma},  
\begin{equation} \label{large on a large set}
\Bigl|\Bigl\{x \in Q : \max_{1 \leq l < 2^{m}} \Bigl| \sum^{l}_{N = 1}
f_{\sigma(N)}(x) \bigr|   \geq   C_{2} m \Bigr\}\Bigr| > C_{3}.
\end{equation} 
\end{enumerate}
\end{prop}
\noindent {\em{Remark: }} The functions $f_N$ given by Proposition \ref{building blocks proposition} do not form a tree system as defined in Section \ref{tree-systems-subsection}. However, there are sequences of functions closely related to Re$(f_N)$ that are in fact tree systems or signed tree systems. We elaborate on these connections in Section \ref{building-block-proof-section} where we prove the proposition; see specifically Lemma \ref{Lemma 4} \eqref{EN-tree-system} and \eqref{fN-tilde}. 
\vskip0.1in 
\begin{lemma} \label{Lemma 9}
Let $U$ be a set of unit vectors in $\RR^n$, all pointing in distinct directions. 
Assume that $\# U = M$ for some $M\geq 2$, and that
all vectors $v\in U$ obey $v\cdot e_n>0$, where $e_n=\left(0, \ldots 0, 1\right)$.
Then there is an ordering $\{u_1,\dots,u_{M}\}$ of vectors in $U$, and a collection of
pairwise disjoint sectors $S_{1}, \ldots S_{M-1}\subset \RR^{n}$ (see Definition \ref{Definition of sector}), such that, up to sets of
Lebesgue measure 0, we have for $l=2,\dots,M$
\begin{equation}\label{sector-inclusion-b}
\Gamma_{u_l} \cap S_i =
 \begin{cases} 
S_i &\text{ if } i < l, \\ \emptyset &\text{ if } i \geq l.
\end{cases} 
\end{equation}
\end{lemma}

\subsection{Completion of the proof} \label{main-proof-section} \begin{proof}[Proof of Theorem \ref{our Theorem2}, assuming
      Proposition \ref{building blocks proposition} and Lemma \ref{Lemma 9}]
 As noted previously, it suffices to prove \eqref{my fond estim changed}. Let $1 < p < \infty$, and let $p_0$ be an integer
such that $p<2p_0$. Assume without loss of generality that $\#U$ is sufficiently large relative to $p_0$, since the bound \eqref{my fond estim changed} is trivial otherwise. 
By rotational symmetry, we may assume (after passing to a subset of cardinality at least $(\#U)/2$ if necessary)
that all vectors $v\in U$ obey $v\cdot e_n>0$, where $e_n=\left(0, \ldots 0, 1\right)$. Passing to a further 
subset $\widetilde{U} \subseteq U$, we may also assume that $\#\widetilde{U} = 2^{m}$ with $m\in\NN$ and $m\geq 0.1 \log (\#U)$. Since $T_U$ dominates $T_{\widetilde{U}}$, we will henceforth work with $\widetilde{U}$, renaming it $U$. 
\vskip0.1in
\noindent Lemma \ref{Lemma 9} now yields an ordering $\{u_1,\dots,u_{2^m}\}$ of vectors in $U$, and a collection of
non-empty and pairwise disjoint sectors $S_{1}, \ldots S_{2^{m}-1}\subset \RR^{n}$, such that
for $l=2,\dots,2^m$ we have
\begin{equation}\label{sector-inclusion}
\Gamma_{u_l} \cap S_i =
 \begin{cases} 
S_i &\text{ if } i < l, \\ \emptyset &\text{ if } i \geq l. 
\end{cases} 
\end{equation}
We now apply Proposition \ref{building blocks proposition} to
the sectors $X_{N} := S_{\sigma^{-1}(N)}$ for $N = 1, \ldots 2^{m}-1$. Let $f=\sum_{N=1}^{2^m-1} f_N$, where
$f_N$ are the functions provided by Proposition \ref{building blocks proposition}. By (\ref{f-pnorm}), we have
\begin{equation} \label{f-pnorm2}
   \|f\|_p \leq C_{1} \sqrt{m}. 
\end{equation} 
On the other hand, by Proposition \ref{building blocks proposition}(a), we have 
$\supp \widehat{f}_{\sigma(N)}  \subset X_{\sigma(N)}= S_{N} $. Using this and (\ref{sector-inclusion}), we get 
that for $l=2,\dots,2^{m}$,
\begin{align*}
T_{u_l}f &= (\chi_{\Gamma_{u_l}} \widehat{f})^\vee
= \Big( \sum_{N=1}^{2^m-1} \chi_{\Gamma_{u_l}} \widehat{f_{\sigma(N)}} \Big)^\vee
= \Big( \sum_{N=1}^{l-1} \widehat{f_{\sigma(N)}} \Big)^\vee
= \sum_{N=1}^{l-1} {f_{\sigma(N)}}.
\end{align*}
Hence
$$
T_U f(x) \geq \max_{2\leq l \leq 2^m}T_{u_l}f 
=\max_{1\leq l \leq 2^m-1}  \Big| \sum_{N =1}^{l} {f_{\sigma(N)}} (x)\Big|.
$$
By Proposition \ref{building blocks proposition}(c), it follows that
$$
|\{x\in Q: T_Uf(x) \geq C_2 m\}| > C_3,
$$
so that for any $1 < p < \infty$, we have
$$
\|T_Uf\|_p \geq C_2 C_3^{1/p} m.
$$
The estimate \eqref{my fond estim changed} follows from this and (\ref{f-pnorm2}).
 \end{proof} 

\section{Proof of Proposition \ref{building blocks proposition}} \label{building-block-proof-section}
\subsection{The inductive construction of functions} \label{construction-section} Proposition
\ref{building blocks proposition} asserts the existence of certain
functions $f_N$; these will be of the following form,
\begin{equation} \label{def-fN} f_N(x) := e(\bar{p}_N \cdot x) g_N(x) \quad
  \text{ with } \quad g_N := \phi_{\ell_N} \ast \chi_{E_N}. \end{equation}  
We pause for a moment to clarify the notation in the preceding
line. Here $\phi_{\ell}(x) := \ell^{n}\phi(\ell x)$, and
$\phi$ is a Schwartz function on $\RR^{n}$ such that
\begin{eqnarray*}
\phi \geq 0, \; \int_{\mathbb{R}^{n}} \phi(x)dx = 1 \; \text{ and } \;  \text{supp}(\widehat{\phi})\subset \left[-1,1\right]^{n}. 
\end{eqnarray*} The sets
$E_N$, the parameters $\ell_N$ and the vectors $\bar{p}_N$ appearing
in \eqref{def-fN} will be specified shortly in Proposition \ref{Lemma 3} below using an inductive mechanism and in the sequential order \[ E_1 \rightarrow
\ell_1 \rightarrow \bar{p}_1 \rightarrow E_2 \rightarrow \ell_2
\rightarrow \bar{p}_2 \rightarrow \cdots, \]  subject to the defining
condition $E_1 := Q = [0,1]^n$, and
\begin{equation}\label{def-esets}
E_{N}  =  E^{(k)}_{j} := \big\{ x \in E_{N^{\ast}} : (-1)^{j-1}  \cos \left( 2\pi x \cdot \bar{p}_{N^{\ast}} \right) > 0 \big\},
\quad N\geq 2, \quad N^{\ast} = \lfloor \frac{N}{2}\rfloor. 
\end{equation} 
As we will see, the parameter $p_N$ specifies the ``location'' and $\ell_N$ the ``size'' of the frequency support of $f_N$. These frequency supports will obey a number of constraints, one of which is pairwise disjointness. On the other hand, the spatial support of $f_N$, while not perfectly localized, is essentially contained in $E_N$. The set $E_N$ will be shown to be nonempty and of positive measure, for every $N$. Here for sets as well as functions, we will continue to use the
double-indexing notation from Section \ref{tree-section}, identifying $N$ with the pair $(k,j)$ as given by the relation \eqref{e-jk}. 
We will also use $Q(\ell, x)$ to
denote the axis-parallel cube with centre $x$ and side length $\ell$. For a multi-index $J=(j_1,\dots,j_h)\in \mathbb{N}^h$,
we will write $\|J\|_\infty=\max_i j_i$; additionally, if $J=(J_1,J_2)$ is a pair of such multi-indices, we will use
$\|J\|_\infty$ to denote $\max(\|J_1\|_\infty,\|J_2\|_\infty)$. 

\vskip0.2in

\begin{prop} \label{Lemma 3} 
Let $p_{0}$ and  $\{X_{N}: N = 1, 2, \ldots 2^{m} - 1\}$ be as in
Proposition \ref{building blocks proposition}. For any sufficiently
large $C\gg p_{0}$, there exists a choice of  large constants
$\ell_{N}$ and vectors $\bar{p}_{N} \in \mathbb{Z}^{n}$ of large
magnitude such that for all $N=1,\dots,2^m-1$ the following properties
hold. 
\begin{enumerate}[(a)]
\item \label{in-sector} $Q_{N}   :=   Q  \left(\ell_{N}, \bar{p}_{N}\right)  \subset
  X_{N}$.
\vskip0.1in 
\item \label{intersection} Given any $1 \leq h \leq p_0$, and any $2h$-tuple of
  indices $J = (J_1, J_2) \in \{1, \ldots, N\}^{2h}$, with $J_1 =
  (j_1, \ldots, j_h)$, $J_2 = (j_1', \ldots, j_{h}')$, we have 
\begin{equation} \label{cubespread}  
\sum_{j \in J_1} Q_j \cap \sum_{j \in J_2} Q_j = \emptyset \quad \text{
  whenever } \quad  \#\{r: j_r = ||J||_{\infty} \} \neq
\#\{r: j_r' = ||J||_{\infty} \}.
\end{equation} 
Here the sum of sets denotes the Minkowski sum, where $A + B = \{a + b : a \in A, b \in B \}$. 
\vskip0.1in 
\item \label{pN-int} For $E_N$ defined as in \eqref{def-esets}, the vector
  $\bar{p}_N$ additionally satisfies   
\begin{equation} \label{ENpN-integral}\int_{E_{N}} \left|\cos \left( 
      2\pi x \cdot \bar{p}_{N}
    \right)\right|  dx  \,\,  >  \,\,  \frac{\left| E_{N}
    \right|}{3}.
\end{equation} 
\item \label{gN-conditions}The functions $\chi_N := \chi_{E_N}$ and $g_N$ in \eqref{def-fN} obey 
\begin{align} &0 \leq  g_{N} \leq 1, \quad
 \text{supp}(\widehat{g}_{N})\subset \left[-\ell_{N} , \ell_{N}\right]^{n},
 \text{ and } \nonumber \\
\label{smallerrors} &\big\| g_{N} - \chi_{{N}} \big\|_1 + \big\| g_{N} - \chi_{{N}} \big\|_{2p_0}  \leq  {2^{-Cm}} .
\end{align}
\vskip0.1in
\item \label{nonempty-condition} The function $x \mapsto \cos(2 \pi x \cdot \bar{p}_N)$ changes sign in $E_N$. More precisely, the sets $\{x \in E_N : \cos(2 \pi x \cdot \bar{p}_N) > 0  \}$ and  $\{x \in E_N : \cos(2 \pi x \cdot \bar{p}_N) < 0  \}$ both have positive Lebesgue measure. 
\end{enumerate}
\end{prop}
\vskip0.2in
\noindent {\em{Remark:}} Before embarking on the proof, let us rephrase the
geometric condition \eqref{cubespread} in an analytical form that is
more convenient to check. Since 
\[ \sum_{j \in J_i} Q_j = Q(L_i, \bar{P}_i), \quad \text{with} \quad
L_i = \sum_{j \in J_i} \ell_j, \; \bar{P}_i = \sum_{j \in J_i}
\bar{p}_j, \]
the condition \eqref{cubespread} is equivalent to $\bar{P}_1 -
\bar{P}_2 \not\in Q(L_1 + L_2, 0)$. If we set $j_0 = ||J||_{\infty}$, $\mu = \#\{r : j_{r} = j_0\}$ and  $\nu = \# \{r : j'_{r} = j_0\}$, this in turn can be written as  
\begin{equation}\label{cubespread2}
 (\mu - \nu)\bar{p}_{j_0}  + \sum_{j_{r} < j_{0}} \bar{p}_{j_{r}} -
\sum_{j'_{r} < j_{0}}\bar{p}_{j'_{r}}   \notin    Q \big( \sum_{j \in
  J} \ell_{j} , 0 \big).
\end{equation} 
If $\mu \ne \nu$, this condition specifies a set of possible
$\bar{p}_{j_0}$ that ensures the disjointness condition \eqref{cubespread}. We will use this to define $\bar{p}_{N}$ in the sequel. 

\vskip0.2in

\begin{proof}[Proof of Proposition \ref{Lemma 3}] 
The proof proceeds by induction on $N$. The sequence $\{\phi_{\ell} :
\ell \geq 1\}$ is an approximation to the identity; hence setting $E_{1} = E^{(0)}_{1} =
Q$, we can choose $\ell_{1} > 0$ large enough so that
(\ref{smallerrors}) holds with $N=1$. Clearly $0\leq g_1\leq 1$.
Further $\widehat{g}_{1} = 
\widehat{\phi}_{\ell_{1}} \widehat{ \chi}_{E_{1}}$, so we also have 
$\supp(\widehat{g}_{1})\subset \supp \widehat{\phi_{\ell_{1}}}
\subset \left[-\ell_{1} , \ell_{1}\right]^{n}$. This verifies the
requirements of part (\ref{gN-conditions}). The condition \eqref{cubespread} (or equivalently
\eqref{cubespread2}), as required by part
(\ref{intersection}), is vacuous in this case, since the only cube
available so far is $Q_1$, and hence $\mu = \nu$ for any choice of
multi-index $J$. For \eqref{pN-int}, we observe that 
for any choice of nonzero $\bar{p}_{1}\in \mathbb{Z}^{n}$ we have
\begin{align*}
\int_{E_1} \left|\cos \left( 2\pi x \cdot \bar{p}_{1} \right)\right|
\, dx &=\int_{Q}  \left|\cos \left( 2 \pi x \cdot \bar{p}_{1} \right)\right| \, dx   
\geq  \int_{Q}  \cos^{2} \left( 2 \pi x \cdot \bar{p}_{1} \right)  dx
\\
&= \int_{Q}  \frac{1 + \cos \left( 4 \pi x \cdot \bar{p}_{1} \right)}{2} dx   
=  \frac{|Q|}{2}=\frac{\left|E_{1}\right|}{2} >  \frac{\left|E_{1}\right|}{3}. 
\end{align*}
Thus any nontrivial choice of $\bar{p}_1$ would ensure
\eqref{ENpN-integral}. Since $E_1 = Q$, condition \eqref{nonempty-condition}  is also trivially satisfied. With $\ell_1$ already chosen as above and keeping in mind that $X_1$ is a sector with unbounded
interior, we can now select $\bar{p}_{1}$ so that \eqref{in-sector} holds. This
completes the verification of the base case $N = 1$.
\vskip0.1in
\noindent For the inductive step, assume that we have constructed $\bar{p}_i,l_i,E_i, g_i$ obeying all conclusions of the lemma for $i=1, 2, \ldots , N-1$. Define $E_N$ via (\ref{def-esets}). Note that this is possible since
$E_{N^{\ast}}$ and $\bar{p}_{N^{\ast}}$ have already been set. Further, $E_N$ thus defined is nonempty, measurable and of positive measure since condition \eqref{nonempty-condition} holds for $N^{\ast}$. Hence we can choose $\ell_N>0$ large enough so
that (\ref{smallerrors}) holds. The properties $0 \leq  g_{N} \leq 1$
and supp$(\widehat{g}_{N})\subset \left[- \ell_{N} , \ell_{N}\right]^{n}$ follow as
in the case $N=1$, establishing part \eqref{gN-conditions}. 
For \eqref{pN-int}, we argue as follows: given any $\bar{p} \in \mathbb{Z}^{n}$ we have
\begin{align*}
 \int_{E_{N}}  \left|\cos \left( 2\pi x \cdot \bar{p} \right)\right|  dx  
&\geq  \int_{E_{N}}  \left|\cos \left( 2\pi x \cdot \bar{p}
  \right)\right|^2  dx   \geq \int_{E_{N}}  \frac{1 + \cos \left( 4 \pi x \cdot \bar{p} \right)}{2} dx  \\
&= \frac{\left|E_{N}\right|}{2}  + \frac{1}{2}\int_{E_{N}}  \cos
\left( 4 \pi x \cdot \bar{p} \right) dx = \frac{\left|E_{N}\right|}{2}
+ \frac{1}{2}\hbox{Re\,}\widehat{\chi}_{N}(2\bar{p}).
\end{align*}
By the Riemann-Lebesgue lemma, $\widehat{\chi}_{N}(\xi) \rightarrow 0$ as $\left|\xi\right| \rightarrow \infty$. 
Thus for any choice of $\xi = \bar{p}_N$ with $|\bar{p}_N|$ large
enough, we can ensure $|\widehat{\chi}_{N}(\bar{p}_N)| < |E_N|/3$, resulting in
\eqref{ENpN-integral}. 
\vskip0.1in
\noindent For part \eqref{intersection}, we must choose $\bar{p}_{N}$ so that 
(\ref{cubespread2}) holds for all $2h$-dimensional multi-indices $J = (j_1, \cdots, j_h; j_1', \cdots, j_h')$
with $||J||_{\infty} \leq N$ and $\mu \ne \nu$. If $||J||_{\infty} \leq N-1$, this is a consequence of the induction hypothesis. 
We may therefore assume $||J||_{\infty} = N$. This means that $j_0 = N$ in the notation of \eqref{cubespread2}. In order to ensure \eqref{cubespread2}, we must have for $s=\pm1, \dots,\pm p_0$,
$$
s \bar{p}_{N} \notin 
- \sum_{j_{r} < N} \bar{p}_{j_{r}} + \sum_{j'_{r} <
  N}\bar{p}_{j'_{r}}  + Q \big( \sum_{j \in J} \ell_j, 0 \big).
$$
Since $\bar{p}_1, \cdots, \bar{p}_{N-1}$ and $\ell_1, \cdots, \ell_N$
have been determined by the previous steps of the construction, the
right hand side of the relation above gives us a finite number of
known cubes that $s \bar{p}_{N} $ must avoid for  $s=\pm1, \dots,\pm p_0$. 
This can be guaranteed if we assume that $|\bar{p}_{N}|$ is large
enough. 
\vskip0.1in
\noindent To establish \eqref{nonempty-condition}, we observe that the periodic function $x \mapsto \cos(2 \pi x \cdot \bar{p}_N)$ alternately assumes positive and negative values on parallel strips separated by distance $\sim |p_N|^{-1}$ and of comparable thickness. Thus given any open ball in $Q$, one can always choose $\bar{p}_N$ large enough so that $\cos(2 \pi x \cdot \bar{p}_N)$ changes sign on the ball. Since $E_N$ is by definition open relative to $Q$, condition \eqref{nonempty-condition} follows. 
\vskip0.1in
\noindent Note that the possible choices of $\bar{p}_N$ so far only
require the vector to be large in magnitude, with no restriction in
direction. Now we choose a specific direction, and place $\bar{p}_{N}$
so that we additionally have $Q_{N} \subset X_{N}$, establishing \eqref{in-sector}. This completes the
inductive step and hence the proof of the proposition.
\end{proof}

\subsection{Finer properties of $f_N$} The algorithm described in
Section \ref{construction-section} endows the resulting sets $E_N = E_j^{(k)}$ and
functions $f_N = f_j^{(k)}$ with properties beyond those given in Proposition
\ref{Lemma 3}. A few of these finer properties are essential to the proof of Proposition \ref{building
  blocks proposition}. We record them here.    
   
\begin{lemma} \label{Lemma 4}
Let $E_{N}$ be the sets defined in Proposition \ref{Lemma 3}. Then the
following conclusions hold. 
\begin{enumerate}[(a)]
\item \label{EN-tree-system} The functions $\chi_N= \chi_j^{(k)}$ form
a tree system, as defined in Definition \ref{tree-def}(\ref{tree-system-def})). 
In particular, they obey the conclusions of Lemma
\ref{lemma-nesting}(a)-(c). Moreover, the following functional identities hold Lebesgue almost surely on $Q$. 
\begin{align}\label{sum-1step}
\chi^{(k+1)}_{2j-1} +\chi^{(k+1)}_{2j}  &=\chi ^{(k)}_{j}, \text{ and }\\ 
\label{sum-m-step}
\sum^{2^{m} - 1}_{N = 1}   \chi_N &\equiv m 
\end{align}
\vskip0.1in
\item \label{sum-over-ht} Let $N_{0} \in \left\{1, \ldots 2^{m}-1\right\}$. Using the
  terminology of trees introduced in Section \ref{tree-section}, let us denote by $\mathcal{T}^{(N_{0})}$ the subtree of $\mathcal{T}_{m}$ having $N_{0}$ as root. Then for every fixed integer $r$ with  $h(N_{0}) \leq r \leq m$, we have:
\begin{equation}\label{sum-subtrees-step}
 \sum_{\substack{N \in \mathcal{T}^{(N_{0})} \\ h(N) = r}}  \chi_{N} = \chi_{N_{0}}.
\end{equation}
\vskip0.1in 
\item \label{fN-tilde} The family of functions $\{\tilde{f}_{N} :=
  \cos(2 \pi \bar{p}_N \cdot) \chi_N : N = 1,
  \cdots, 2^m-1\}$ is a signed tree system in the sense of Definition \ref{tree-def}(\ref{signed-tree-system-def}).
\end{enumerate} 
\end{lemma}


\begin{proof} Rewriting \eqref{def-esets} in terms of $j,k$, we get that $E_1=Q$ and
\begin{align*}
E^{(k+1)}_{2j-1} &:= \big\{ x \in E^{(k)}_{j} :  \cos ( x \cdot \bar{p}^{(k)}_{j} ) > 0 \big\}, \\
E^{(k+1)}_{2j} &:= \big\{ x \in E^{(k)}_{j} :  \cos ( x \cdot \bar{p}^{(k)}_{j} ) < 0 \big\}.	
\end{align*}
Therefore, the sets $E^{(k+1)}_{2j-1}$ and $E^{(k+1)}_{2j}$ are disjoint and contained in $E^{(k)}_{j}$,
so that the functions $\chi^{(k)}_j$ form a tree system. 
Since the set $\big\{ x \in Q: \cos ( x \cdot \bar{p}^{(k)}_{j}) = 0
\big\}$ has Legesgue measure 0, we also have
(\ref{sum-1step}). Iterating (\ref{sum-1step}), we get that the
following holds almost surely, 
\begin{equation} \label{chiQ}
\begin{aligned}
\chi_Q =\chi^{(0)}_1 & = \chi^{(1)}_1+ \chi^{(1)}_2
\\
&=  (\chi^{(2)}_1+ \chi^{(2)}_2 ) + (\chi^{(2)}_3+ \chi^{(2)}_4)
\\
& = \dots = \sum_{1\leq j\leq 2^k} \chi^{(k)}_j
\end{aligned}
\end{equation} 
for every $k=0,1,\dots, m-1$. Summing over $k$ yields (\ref{sum-m-step}).
\vskip0.1in
\noindent We now turn to (\ref{sum-subtrees-step}).
If $h(N_{0}) = r$, the summation is over the single vertex $N_0$ and there is
nothing to prove. If $h(N_{0}) < r \leq m$, then \eqref{sum-subtrees-step} follows from the same calculations as in \eqref{chiQ}, except we start from $\chi_{N_0}$ instead of $\chi_Q$. 

\vskip0.1in
\noindent Regarding (\ref{fN-tilde}), let us observe that
$\tilde{f}_{N}(x) = \cos(2 \pi \bar{p}_{N} \cdot x) \chi_N$ and $\cos\left(2\pi\bar{p}_{N} \cdot x\right)$ have
the same sign in the set $E_{N}$. In view of \eqref{def-esets}, this
shows that we have, up to sets of Lebesgue measure zero 
\begin{align*}
\supp \tilde{f}_N   &=   E_{N}   =  \Big\{ x \in E_{N^{\ast}} : (-1)^{j-1}  \cos \left( 2\pi x \cdot \bar{p}_{N^{\ast}} \right) > 0 \Big\}
\\
&=  \Big\{ x \in E_{N^{\ast}} : (-1)^{j-1}  \tilde{f}_{N^{\ast}}(x) > 0 \Big\}  \,\,  =  \,\,  \Big\{ x \in Q : (-1)^{j-1}  \tilde{f}_{N^{\ast}}(x) > 0 \Big\}.
\end{align*}
This is exactly the signed tree system condition (\ref{tree system condition 3}).
\end{proof}

\vskip0.2in
\noindent The confluence of spatial and frequency localization built
into the definition of $f_N$ results in a high degree of orthogonality
amongst them. This interaction is manifested in the $L^p$-norms of
their sums, for large exponents $p$. The following proposition, which offers an estimate of
this norm, is a critical component in the proof of Proposition
\ref{building blocks proposition}(\ref{f-pnorm}). The proof of the proposition is
nontrivial and is relegated to Section \ref{norm-estimate-section}.  
\begin{prop}\label{2p0-norm-lemma} 
For $p_0$ and $\{X_N\}$ as in Proposition \ref{building blocks proposition}, let $\{f_N: N = 1, \ldots, 2^m-1\}$ be the family of
functions given by \eqref{def-fN} in Section \ref{construction-section}. Then
there exists a constant $C_0 > 0$ depending only on $p_0$ and $n$ such
that \begin{equation}  \bigl|\big|
\sum_{N=1}^{2^m-1} f_N||_{2p_0} \leq C_0 \sqrt{m}. \label{2p0-norm}\end{equation}      
\end{prop} 
\noindent Assuming this, the proof of Proposition \ref{building blocks proposition} is completed in
the next subsection. 

\subsection{Proof of Proposition \ref{building blocks
    proposition}}\begin{proof} Since supp$(\widehat{f}_N) \subseteq Q_N$,
  part \eqref{support-cone} of the proposition follows from Proposition \ref{Lemma
    3}(\ref{in-sector}). 
\vskip0.1in
\noindent Let us turn to \eqref{f-pnorm}. Given that $p \in [1, 2p_0]$, the desired conclusion
follows from the log-convexity of Lebesgue norms,  provided we
have the correct estimates at the endpoints $p = 2p_{0}$ and $p = 1$. Proposition
\ref{2p0-norm-lemma} asserts the necessary bound for $p = 2p_0$. Our
claim is that the bound for $p = 1$ follows from the same proposition. The following chain of inequalities establishes this claim: 
\begin{align*}
\big\| \sum^{2^{m} - 1}_{N = 1}  f_{N} \big\|_1 &\leq   \big\|
\sum^{2^{m} - 1}_{N = 1}  e(\bar{p}_{N} \cdot ) \chi_{N}
\big\|_1 +   \big\| \sum^{2^{m} - 1}_{N = 1} e(\bar{p}_{N} \cdot )  \big(  g_N - \chi_{N}  \big)\big\|_1
\\
&\leq \big\| \sum^{2^{m} - 1}_{N = 1}  e(\bar{p}_{N} \cdot )
\chi_{N} \big\|_1 + \sum^{2^{m} - 1}_{N = 1} || g_N  - \chi_{N} ||_1 \\ 
& \leq \big\| \sum^{2^{m} - 1}_{N = 1}  e(\bar{p}_{N} \cdot )
\chi_{N} \big\|_1 + 2^{m-Cm} \\ 
& \leq \big\| \sum^{2^{m} - 1}_{N = 1}  e(\bar{p}_{N} \cdot )
\chi_{N} \big\|_{2p_0} + 2^{m-Cm} \\ 
&\leq \big\| \sum^{2^{m} - 1}_{N = 1} f_{N} \big\|_{2p_0} + \big\| \sum^{2^{m} - 1}_{N = 1}  e(\bar{p}_{N} \cdot )
(\chi_{N} - g_N)\big\|_{2p_0} + 2^{m-Cm} \\
& \leq \big\| \sum^{2^{m} - 1}_{N = 1}  f_{N} \big\|_{2p_0} +
2\cdot 2^{m-Cm} \\
& \leq C_0 \sqrt{m} + 1, \text{ which is } \leq C_1 \sqrt{m} \text{ if
  } C_1 > 2C_0. 
\end{align*}
The third and the sixth inequality in the sequence above uses the
error bound \eqref{smallerrors} proved in Proposition \ref{Lemma 3}(\ref{gN-conditions}). The fourth inequality follows from the fact that
$G = \sum_N e(p_N \cdot) \chi_N$ is supported on $Q$, and hence
$||G||_1 \leq ||G||_{2p_0}$ by H\"older's inequality. The last
inequality follows from the main estimate \eqref{2p0-norm} in
Proposition \ref{2p0-norm-lemma}.  The triangle
inequality is used throughout. This completes the proof of \eqref{f-pnorm}.   
\vskip0.1in
 \noindent It remains to prove (\ref{f-sup-norm}). Recall from Lemma
\ref{Lemma 4}(\ref{fN-tilde}) that $\tilde{f}_N := \cos(2 \pi
\bar{p}_N \cdot) \chi_N$ is a signed tree system. Hence by Lemma
\ref{Karagulyan's tree systems lemma} for signed tree systems, we have
\begin{align*} 
\max_{1 \leq \, l < 2^{m}}  \big| \sum^{l}_{N = 1} f_{\sigma(N)}(x) \big|  \,\,  &\geq  \,\,  \max_{1 \leq \, l < 2^{m}}  \big| \sum^{l}_{N = 1} \text{Re\,}f_{\sigma(N)}(x) \big|
\\
&\geq  \,\,  \max_{1 \leq \, l < 2^{m}}  \bigl| \sum^{l}_{N = 1} \tilde{f}_{\sigma(N)}(x) \bigr|  -  \mathcal{E}(x)  \\  &\geq  \,\   
\frac{1}{3}  \sum^{2^{m} - 1}_{N = 1} \big|\tilde{f}_{N}(x)\big|  \,\,   -  \mathcal{E}(x),  
\end{align*}
for all $x \in \mathbb{R}^{n}$, where
$\mathcal{E}(x)   \,\,  =  \,\,  \sum^{2^{m} - 1}_{N = 1}  \bigl|\tilde{f}_{N}(x) -  \text{Re\,}f_{N}(x) \bigr|$. 
The last inequality above follows from \eqref{tree system inequality},
the rest from the triangle inequality. We will show that 
\begin{align} 
\Big|\Bigl\{ x : \sum^{2^{m} - 1}_{N = 1}  \big|\tilde{f}_{N}(x)\big| \geq \frac{m}{10} \Bigr\} \Big| &\geq \frac{1}{10}, \quad {\text{ and that}} \label{ftilde-large} \\
\big| \bigl\{x : |\mathcal{E}(x)| > 1 \bigr\} \big| &\leq 
  \frac{1}{100} \text{ if $C$ in Proposition \ref{Lemma 3} is sufficiently large.} \label{E-large} 
\end{align} 
For large $m$, this would ensure \eqref{large on a large
  set} and complete the proof, with constants
$C_2 = 1/20$ and $C_3 = 9/100$, for instance.  
\vskip0.1in
\noindent To prove \eqref{ftilde-large}, let us set $\tilde{f}(x) :=
\sum^{2^{m} - 1}_{N = 1}  \big|\tilde{f}_{N}(x)  \big|$. On one
hand, by Proposition \ref{Lemma 3}(\ref{pN-int}), 
\begin{eqnarray*}
\int_{Q}  \big|\tilde{f}_{N}(x)  \big| dx  \,\,  =  \,\,    \int_{E_{N}}  \left|\cos \left(2\pi x \cdot \bar{p}_{N} \right)\right| dx  \,\,  >  \,\,  \frac{\left|E_{N}\right|}{3}.
\end{eqnarray*}
Summing over all $N$ and using (\ref{sum-m-step}) in Lemma \ref{Lemma
  4}, we obtain 
\begin{equation} \label{1}
\int_{ Q}  \tilde{f}(x) \,dx =\int_{ Q}  \sum^{2^{m} - 1}_{N = 1} 
\big|\tilde{f}_{N}(x) \big| dx  \,\,  \geq  \,\,  \frac{1}{3} 
\sum^{2^{m} - 1}_{N = 1}  \left|E_{N}\right|   =  \frac{1}{3}
\int_{Q} \sum_{N=1}^{2^m-1} \chi_N(x) \, dx =  \frac{m}{3}.
\end{equation}
On the other hand, 
\begin{equation} \label{2} 
\tilde{f}(x)  \,\,  \leq  \,\, \sum^{2^{m} - 1}_{N = 1} 
\chi_{N}(x)  \,\,  =  \,\,  m \text{ for a.e. $x \in
  Q$.} \end{equation} 
Set $E := \{x \in Q : \tilde{f}(x) \geq \frac{m}{10} \}$. Combining \eqref{1} and \eqref{2}, we see that  
\[
m |E| + \frac{m}{10}
\geq \int_{E} 
\tilde{f}(x)dx  + \int_{Q \setminus E}  \tilde{f}(x)dx =  \int_{ Q}  \tilde{f}(x)dx  \,\,  \geq  \,\, \frac{m}{3},
\]
This shows that $|E| > \frac{1}{3} - \frac{1}{10} > \frac{1}{10}$, establishing \eqref{ftilde-large}. 
 \vskip0.1in
\noindent Regarding \eqref{E-large}, we make use of \eqref{smallerrors} to deduce that 
\begin{align*}
\left\| \mathcal{E} \right\|_1 =   \big\|\sum^{2^{m} - 1}_{N= 1}  \bigl| \tilde{f}_{N} -  \text{Re\,}f_{N} \bigr| \big\|_1  
&\leq  \big\|\sum^{2^{m} - 1}_{N= 1}  \bigl| \cos(2 \pi \bar{p}_N \cdot) (\chi_N - g_N )    \bigr| \big\|_1 \\ &\leq \sum_{N=1}^{2^m-1} ||\chi_N - g_N||_1 \leq 2^{m - Cm}. 
\end{align*}
Therefore, by Chebyshev's inequality
\begin{equation}\label{3}
\bigl| \big\{ x \in Q :  \mathcal{E}(x) > 1 \big\} \bigr|  \,\,  \leq  \,\, \bigl\| \mathcal{E} \bigr\|_1  \leq 2^{m-Cm}, 
\end{equation}
which proves our claim \eqref{E-large} for $C > 0$ sufficiently large. 
\end{proof}

\section{Norm estimate: Proof of Proposition \ref{2p0-norm-lemma}} \label{norm-estimate-section} 
\noindent This section is given over to the estimation of the $L^{2p_0}$ norm of
the function $f := \sum_N f_N$, with the summands $f_N$ defined as in
Proposition \ref{Lemma 3}. Parts of the argument are highly
combinatorial, involving summations over index sets whose members are long
sequences of integers. Two previously introduced tools will continue
to be useful for book-keeping purposes; namely, the double-indexing notation relating $N$ with
the pair $(k,j)$ as in \eqref{e-jk}, and the language of trees as described in Section
\ref{tree-section}. We begin by setting up some supplementary notation that will be
convenient for handling sums over large index sets later on.      
\vskip0.1in
\subsection{Notation} 
\subsubsection{Small errors} \label{errors-section} For any two quantities $X$ and $Y$
depending on $m$, we will write $X  \cong  Y$ if $|X - Y| \leq A 2^{-Bm}$, where the multiplicative constant $A$ and the
exponent $B$ may depend on $p_0$ and $n$, and may change from line to line but remain independent of
$m$. Both $A$ and $B$ will always be sufficiently large. In our
applications, $B$ will depend on the large constant $C$ from
Proposition \ref{Lemma 3}. Assuming that $C \gg p_{0}$ was chosen
large enough, we will always be able to ensure that $B > \frac{C}{10}$.
\vskip0.1in
\noindent The notation $X = O(Y)$ will be used to mean $|X| \leq A|Y|$, with the same conditions on the constant $A$  as above. 


\subsubsection{Grouping of vectors of vertices} 
\noindent 
Our main estimate will be proved by expanding the $L^{2p_0}$ norm of $f$ as a sum of integrals
of the form
\begin{equation}\label{eff-eq}
\int f_{m_1}^{\mu_1} \cdots f_{m_r}^{\mu_r}
\overline{ f_{n_1}^{\nu_1} \cdots f_{n_s}^{\nu_s}} dx
\end{equation}
with $\mu_1 + \cdots + \mu_r =\nu_1 + \cdots + \nu_s = p_0$, then grouping these integrals appropriately to obtain 
cancellations and simplifications. The notation introduced in this subsection will facilitate that process.

\vskip0.1in
\noindent Given an integer exponent $1 \leq p \leq p_0$ and an integer dimension $1 \leq r \leq p$, we define a
{\em{multiplicity vector for the exponent $p$ of length $r$}} to be of the form \[\bar{\mu} =
(\mu_1, \cdots, \mu_r) \in \mathbb N^r, \quad \text{ where } \quad \mu_1 + \mu_2 + \cdots
+ \mu_r = p. \]  
The use of a multiplicity vector allows us to rewrite a $p$-long
integer vector with some possibly coincident entries in ``collapsed
form''. For instance, all such sequences with $r$ distinct entries,
where the $i$-th smallest element occurs with frequency $\mu_i$, can
be gathered into a single collection, as explained 
below. \vskip0.1in
\noindent Given integers $1 \leq h \leq m$, $1 \leq r \leq p \leq p_0$ we set 
\begin{equation}
\mathbb A_{p,h}[r] := \Bigl\{ (\bar{m}, \bar{\mu}) \Bigl| \;
\begin{aligned} 
&\bar{m} = (m_1, \cdots, m_r) \in \mathbb N^r, \;1 \leq m_1 < \cdots < m_r < 2^h \\ 
& \bar{\mu} \text{ is a multiplicity vector for $p$ of length $r$} 
\end{aligned}  
\Bigr\}.  \label{def-grandA} 
\end{equation}
Observe that for every $(\bar{m}, \bar{\mu}) \in \mathbb A_{p,h}[r]$, there exist $p$-dimensional vectors $ \mathbf N \in \mathbb N^p$ such that $m_i$ occurs in the string $\mathbf N$ exactly $\mu_i$ times for each $1 \leq i \leq r$. For example, we can take  $\mathbf N$ to be $\mathbf N[\bar{m}, \bar{\mu}]$, which is by definition a $p$-long vector  
 whose first $\mu_1$ entries are $m_1$, the next $\mu_2$ entries are $m_2$, and so on. The relevance of $\mathbb A_{p, h}[r]$ lies in the following partition of the index set: 
\begin{equation} \label{partition} \{1, \cdots, 2^h-1 \}^p = \bigsqcup_{r=1}^p \{\mathbf N : \exists (\bar{m}, \bar{\mu}) \in \mathbb A_{p,h}[r] \text{ so that $\mathbf N$ is a permutation of $\mathbf N[\bar{m}, \bar{\mu}]$}  \}.  \end{equation}  
\vskip0.1in
\noindent Right now, an element of $\mathbb A_{p,h}[r]$ is a 2-tuple $(\bar{m}, \bar{\mu})$, whose first component $\bar{m}$ is a multi-index and whose second component $\bar{\mu}$ is an $r$-long multiplicity vector for $p$. The number of choices of $\bar{\mu}$ for a fixed $p$ and $r$ is bounded by a constant depending only on $p$ and independent of $\#U$ (and hence $h$), whereas $\bar{m}$ ranges over an index set of cardinality $O(2^{hr})$, which is typically much larger. For the quantitative bounds that we seek, it is therefore no loss of generality to work with a fixed multiplicity vector $\bar{\mu}$ at a time. In order to keep track of the collection of all multi-indices $\bar{m}$ that generate elements of $\mathbb A_{p,h}[r]$ for a fixed multiplicity, we define 
\begin{equation} 
\mathbb A_h[r, \bar{\mu}] := \left\{ \bar{m} : (\bar{m}, \bar{\mu}) \in \mathbb A_{p, h}[r] \right\}, \label{def-Ah} 
\end{equation} 
which is in effect the $\bar{\mu}$-fibre of $\mathbb A_{p,h}[r]$.  
\vskip0.1in
\noindent We will also need to stratify pairs of vectors according to the
position of their combined maximal element in the binary tree. With that in mind and given multiplicity vectors $\bar{\mu}, \bar{\nu}$ of length $r, s$ for the exponents $p,q$ respectively, we set  
\begin{align} 
\mathbb M_h[r, \bar{\mu}; s, \bar{\nu}] &:= \left\{\bar{\alpha} =
(\bar{m}, \bar{n}) \Bigl|\; \begin{aligned} &\bar{m} \in \mathbb A_h[r, \bar{\mu}], \; \bar{n}
\in \mathbb A_h[s, \bar{\nu}], \\ &h(||\bar{\alpha}||_{\infty}) =
h(\max(m_r, n_s)) = h \end{aligned} \right\}. \label{Mh-def}
\end{align} 
Recall that $h(N) = k$ denotes the height of the vertex $N = 2^k+j-1$ in the binary tree $\mathcal T_m$. 
The parameters $r, s, \bar{\mu}, \bar{\nu}$
occurring in the argument of \eqref{Mh-def} will be suppressed if they are clear from the context.

\vskip0.1in
\noindent
Two special subclasses of
$\mathbb M_{h}$ will be important for our analysis. They are: 
\begin{align}
\mathbb M_h' &:= 
\left\{\bar{\alpha} \in \mathbb M_h : \exists \text{ a ray $\mathscr
    R$ in $\mathcal T_h$ such that all entries of $\bar{\alpha}$ lie
    on $\mathscr R$}   
\right\} \\
\mathbb M_h^{\ast} &:= \begin{cases} 
\{\bar{\alpha} = (\bar{m}, \bar{n})\in \mathbb M_h': m_r = n_s \}
&\text{ if } \mu_r = \nu_s, \\ \emptyset &\text{ otherwise}. 
\end{cases} 
\end{align} 
Figures \ref{fig5} and \ref{fig5-star} depict examples of multi-indices $\bar{\alpha}$ that lie in these special subclasses.  

\begin{figure}[htbp]\centering\includegraphics[width=4.in]{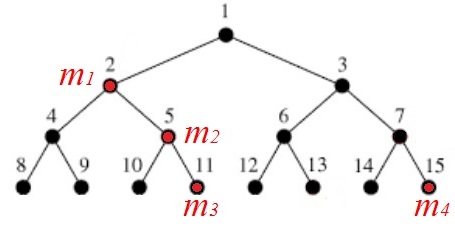}
\caption{Fix $\bar\mu, \bar\nu$ with $r=s=3$.
Let $\bar{m} = (m_1, m_2, m_3)$, $\bar{m}' = (m_1, m_2, m_4)$. Then $\bar{\alpha} = (\bar{m}, \bar{m}) \in \mathbb M_h'$ but $\bar{\beta} = (\bar{m}, \bar{m}') \notin \mathbb M_h'$. 
}
\label{fig5}
\end{figure}

\begin{figure}[htbp]\centering\includegraphics[width=3.in]{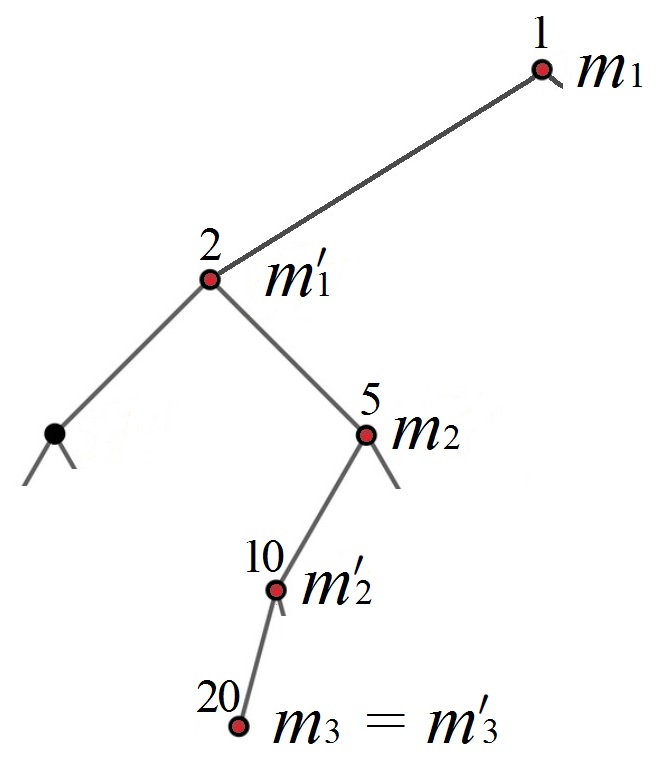}
\caption{Fix multiplicity vectors $\bar\mu, \bar\nu$ with $r=s=3$ and $\mu_3=\nu_3$.
Let $\bar{m} = (m_1, m_2, m_3)$, $\bar{n} = (m_1', m_2', m_3')$ so that $m_3=m'_3$ and all $m_i,m'_i$ lie on the same ray of
the tree. Then $\bar{\alpha} = (\bar{m}, \bar{n}) \in \mathbb M_h^{\ast}$ for these multiplicity vectors. 
}
\label{fig5-star}
\end{figure}

\subsection{Main steps} The relevance of the aforementioned notation in the
context of the norm estimation problem is clarified in the following sequence of
lemmas, which provides the key ingredients. 
\begin{lemma} \label{sum-prep-lemma}
Let $f := \sum_N f_N$, with $f_N$ given by \eqref{def-fN} in Section \ref{building-block-proof-section}. Then  
\begin{equation}  \label{sum with mult} 
||f||_{2p_0}^{2p_0} = \sum' C(\bar{\mu}, \bar{\nu}) \sum_{h=1}^{m} \mathtt F_h[r,\bar{\mu}; s, \bar{\nu}], \quad \text{
  where } \quad \mathtt F_h :=
\sum_{\bar{\alpha} \in \mathbb M_h} \int F_{\bar{\alpha}}(x) \, dx.  
\end{equation} 
To explain the notation in the above line, 
\begin{itemize} 
\item The outer sum $\sum'$ ranges over all choices of positive
integers $1 \leq r, s \leq p_0$ and all choices of
multiplicity vectors $\bar{\mu}, \bar{\nu}$ for the exponent $p_0$ of lengths $r$ and $s$
respectively.
\vskip0.1in  
\item The constants $C(\bar{\mu}, \bar{\nu})$ depend only on
$\bar{\mu}, \bar{\nu}, r, s, p_0$ and are independent of $m$; specifically
\[ C(\bar{\mu}, \bar{\nu}) = {{p_0}\choose{\bar{\mu}}} {{p_0}\choose{\bar{\nu}}}  = \frac{(p_0!)^2}{\mu_1! \cdots \mu_r!
  \nu_1! \cdots \nu_s!}.  \] 
 \vskip0.1in
\item Given $\bar{\alpha} = (\bar{m}, \bar{n}) \in \mathbb
M_h[r, \bar{\mu}; s, \bar{\nu}]$, 
\begin{equation}  F_{\bar{\alpha}} := f_{m_1}^{\mu_1} \cdots f_{m_r}^{\mu_r}
\overline{ f_{n_1}^{\nu_1} \cdots f_{n_s}^{\nu_s}}.  \label{def-Falpha} \end{equation}    
\end{itemize} 
\end{lemma}  
\noindent We will continue to use the notation \eqref{def-Falpha} even if $\bar{\mu}, \bar{\nu}$ are multiplicity vectors for different exponents.   
\vskip0.2in 
\begin{lemma} \label{Lemma 5}
Fix $1 \leq h \leq m$, and two multiplicity vectors $\bar{\mu}, \bar{\nu}$ for integer exponents $1 \leq p, q \leq p_0$ of lengths $r, s$ respectively. Then the following conclusions hold: 
\vskip0.1in 
\begin{enumerate}[(a)]
\item \label{FsameG}For any $\bar{\alpha} = (\bar{m}, \bar{n}) \in
\mathbb M_h[r, \bar{\mu}; s, \bar{\nu}]$,   
\begin{align} \int F_{\bar{\alpha}} dx &\cong \int G_{\bar{\alpha}} dx,
  \text{ where } \label{FsameG-estimate} \\ 
G_{\bar{\alpha}}  (x)   &:=   e( v_{\bar\alpha} \cdot x)
\chi_{E_{\bar{\alpha}}} (x), \text{ with }  v_{\bar\alpha} :=\sum_{i=1}^r\mu_{i} \bar{p}_{m_{i}} - \sum_{i=1}^s \nu_{i} \bar{p}_{n_{i}},  \text{ and } \\
E_{\bar{\alpha}} &:= E_{m_{1}} \cap \ldots E_{m_{r}} \cap E_{n_{1}} \cap \ldots \cap  E_{n_{s}}. \label{E-alpha-def}
\end{align} 
Here the notation $\cong$ denotes equality up to small errors, as explained in Section \ref{errors-section}.  The sets $E_N$ on the right hand side of \eqref{E-alpha-def} are as in \eqref{def-esets}. 
\vskip0.1in
\item \label{E-nonempty} If $E_{\bar{\alpha}} \neq \emptyset$, then $\bar{\alpha} \in
  \mathbb{M}'_{h}$. The converse is also true. For such $\bar{\alpha}$, we have \begin{equation} \label{set-vector-number} E_{\bar{\alpha}} = E_{||\bar{\alpha}||_{\infty}}. \end{equation} 
\vskip0.1in
\item \label{F-G-G'} As a consequence, $\mathtt F_h \cong \mathtt G_h = \mathtt
  G'_h$, where 
\[ \mathtt G_h[r, \bar{\mu}; s, \bar{\nu}] := \sum_{\bar{\alpha} \in \mathbb M_h} \int
G_{\bar{\alpha}}(x) \, dx \quad \text{ and } \quad \mathtt G_h'[r, \bar{\mu}; s, \bar{\nu}] := \sum_{\bar{\alpha} \in \mathbb M_h'} \int
G_{\bar{\alpha}}(x) \, dx.\]   
\end{enumerate}
\end{lemma}
\vskip0.2in 
\begin{lemma} \label{Lemma 8}
In the notation of Lemma \ref{Lemma 5}, we have
\begin{equation} \label{Gh-estimate} 
\mathtt G_h[r,\bar{\mu}; s, \bar{\nu}]  \cong \mathtt
G_h^{\ast}[r,\bar{\mu}; s, \bar{\nu}] \cong \left\{
\begin{aligned}
&0 & \text{if } \bar{\mu} \neq \bar{\nu},\\
&O  \left(h^{r-1}\right)  & \text{if }  \bar{\mu} = \bar{\nu}.
\end{aligned} \right.
\end{equation} 
Per our notational convention, we have used \[\mathtt G_h^{\ast} [r, \bar{\mu}; s, \bar{\nu}]:= \sum_{\bar{\alpha} \in \mathbb M_h^{\ast}} \int G_{\bar{\alpha}}(x) \, dx. \] 
 \end{lemma}

\subsection{Proof of Proposition \ref{2p0-norm-lemma}} 
\begin{proof} We complete the proof of the proposition assuming the three lemmas
above. In view of Lemma \ref{sum-prep-lemma}, $||f||_{2p_0}^{2p_0}$ is given by \eqref{sum with mult}. The number of summands in the outer sum on the right hand side of the equation \eqref{sum with mult} depends only on $p_0$; hence it suffices to show that 
\[ \sum_{h=1}^{m} \mathtt F_h[r, \bar{\mu}; s, \bar{\nu}] = O(m^{p_0}) \]  for every fixed choice of integers $r, s \leq p_0$
and for each choice of multiplicity vectors $\bar{\mu}, \bar{\nu}$ for the exponent $p_0$ of lengths respectively $r$ and $s$. Combining Lemma \ref{Lemma 5}\eqref{F-G-G'}
with Lemma \ref{Lemma 8}, we find that 
\[ \sum_{h=1}^{m} \mathtt F_h[r, \bar{\mu}; s, \bar{\nu}] \cong \sum_{h=1}^{m} \mathtt G_h[r,
\bar{\mu}; s, \bar{\nu}]  \cong \sum_{h=1}^{m} O(h^{r-1}) =
O(m^r) = O(m^{p_0}) \text{ if } \bar{\mu} = \bar{\nu}, \]
and $\cong 0$ otherwise, completing the proof.  
\end{proof}

\subsection{Proof of Lemma \ref{sum-prep-lemma}} 
\begin{proof} 
We start by expanding the $L^{2p_0}$-norm of $f$ as follows, 
\begin{eqnarray} \label{norm as sum}
||f||_{2p_0}^{2p_0} = \int \Bigl[\sum_{N} f_N \Bigr]^{p_0}
\overline{\Bigl[\sum_{N}f_N \Bigr]^{p_0}} \, dx =   \int_{\RR^{n}}  \sum_{\mathbf N, \mathbf N'}  
f_{N_{1}} \cdots f_{N_{p_{0}}}\overline{f_{N'_{1}} \cdots f_{N'_{p_{0}}}} dx, 
\end{eqnarray}
where the summation ranges over all $p_0$-dimensional multi-indices
\[ \mathbf N = (N_1, \cdots, N_{p_0}), \; \mathbf N' = (N_1', \cdots,
N_{p_0}') \in \{1, \ldots, 2^{m}-1\}^{p_0}. \] The entries in $\mathbf N$ (and hence also $\mathbf N'$) need not
be distinct. However, in view of \eqref{partition} and the discussion leading up to it, for every $\mathbf N$ there exist 
\begin{itemize} 
\item a unique integer $1 \leq r = r(\mathbf N) \leq p_0$, 
\item a multiplicity vector $\bar{\mu} = \bar{\mu}(\mathbf N)$ of
  length $r$ for $p_0$, and 
\item a choice of $\bar{m} = \bar{m}(\mathbf N) = (m_1, \cdots, m_r) \in \mathbb
A_m[r, \bar{\mu}]$ defined as in \eqref{def-Ah},  
\end{itemize} 
such that $\mathbf N[\bar{m}, \bar{\mu}]$ is a permutation of $\mathbf N$, i.e.,   $m_1 < m_2 < \cdots < m_r$, and $m_i$
occurs exactly $\mu_i$ times in $\mathbf N$. Thus \[f_{N_1} \cdots
f_{N_{p_0}} = f_{m_1}^{\mu_1} \cdots f_{m_r}^{\mu_r}.\] Further, given
a fixed choice of $\bar{\mu}$ and $\bar{m}$, there are exactly
${{p_0}\choose{\bar{\mu}}} = p_0!/(\mu_1! \cdots \mu_r!)$-many
possibilities of $\mathbf N$ that correspond to the same choice of $\mathbf N[\bar{m}, \bar{\mu}]$. Grouping the sum in \eqref{norm
  as sum} using these multiplicities, we obtain 
\begin{equation} \label{norm-rewritten}  ||f||_{2p_0}^{2p_0} = \int \sum'_{\begin{subarray}{c}r, s,
    \bar{\mu}, \bar{\nu} \end{subarray}}
{{p_0}\choose{\bar{\mu}}} {{p_0}\choose{\bar{\nu}}} \sum_{\bar{\alpha}} 
F_{\bar{\alpha}}(x) \, dx,  \end{equation} 
where $F_{\bar{\alpha}}$ has been defined in \eqref{def-Falpha}, the outer sum is as in the statement of the lemma and the inner sum ranges over all multi-indices $\bar{\alpha} = (\bar{m},
\bar{n}) \in \mathbb A_m[r, \bar{\mu}] \times \mathbb A_m[s,
\bar{\nu}]$. Finally, we note that \[ \mathbb A_m[r, \bar{\mu}] \times \mathbb
A_m[s, \bar{\nu}] = \bigsqcup_{h=1}^m \mathbb M_h[r, \bar{\mu} ; s,
\bar{\nu}]. \]
Decomposing the inner sum in \eqref{norm-rewritten} based on $h$ therefore leads to \eqref{sum with mult}. 
\end{proof}  

\subsection{Proof of Lemma \ref{Lemma 5}} 
\begin{proof} Part \eqref{FsameG} of the lemma is based on an iterative
  application of the following estimate: for any measurable function
  $H$ with $||H||_{\infty} \leq 1$, \eqref{smallerrors} gives 
\begin{equation}\label{core estimate} 
\left| \int H(x) (g_N - \chi_N)(x) \, dx \right| \leq ||g_N -
\chi_N||_1 \leq 2^{-Cm}. 
\end{equation}
We use this estimate to successively peel away each factor $f_N$
occurring in $F_{\bar{\alpha}}$, replacing it by $e(p_N \cdot)\chi_N$
instead. Specifically, starting with any $\bar{\alpha} = (\bar{m},
\bar{n}) \in \mathbb M_h[r, \bar{\mu}; s, \bar{\nu}]$, we can
write \[ F^{[1]} :=  F_{\bar{\alpha}} = f_{m_1} F^{[2]} = e(p_{m_1}
\cdot ) g_{m_1} F^{[2]}, \quad \text{ where
} \quad F^{[2]} = f_{m_1}^{\mu_1-1} \cdots f_{m_{r}}^{\mu_r}
\overline{f_{n_1}^{\nu_1} \cdots f_{n_s}^{\nu_s} }\] 
is a product of $p+q-1$ factors. As a result, 
\begin{equation} \label{hacking away}
\begin{aligned} 
\int_{\RR^{n}}  F_{\bar{\alpha}}(x) dx  &= \int_{\mathbb R^n}
F^{[1]}(x) \, dx = \int_{\mathbb R^n} f_{m_1}(x) F^{[2]}(x) \, dx \\
&= \int_{E_{m_1}} e(\bar{p}_{m_1}\cdot x) F^{[2]}(x) \, dx + \int
e(\bar{p}_{m_1}\cdot x) (g_{m_1} - \chi_{m_1})(x) F^{[2]}(x) \, dx \\
&\cong \int_{E_{m_1}} e(\bar{p}_{m_1}\cdot x) F^{[2]}(x) \, dx. 
\end{aligned}
\end{equation} 
The last step above uses \eqref{core estimate} with $H = e(p_{m_1}
\cdot ) F^{[2]}$, which is bounded by 1 according to Proposition
\ref{Lemma 3}\eqref{gN-conditions}.  
Iterating the argument in \eqref{hacking away} exactly $\mu_{1} + \ldots +
\mu_{k} + \nu_{1} + \ldots \nu_{l} = p+q$ times (and using \eqref{core
  estimate} with a different choice of $H$ at each stage), we are able to remove all factors
$f_N$ and are left with the integrand $G_{\bar{\alpha}}$. This is the
desired claim \eqref{FsameG-estimate}. 
\vskip0.1in
\noindent Regarding \eqref{E-nonempty}, we recall from Lemma \ref{Lemma
  4}\eqref{EN-tree-system} that the family of functions $\{\chi_N :=
\chi_{E_N}\}$ is a tree system. In particular, for any two indices $N
< N'$, the sets $E_N$ and $E_{N'}$ (which are non-empty by Proposition \ref{Lemma 3}) are either disjoint or
nested. Their intersection is nonempty precisely when $E_N
\supseteq E_{N'}$, 
which in turn happens if and only if $N$ is an ancestor of
$N'$, when represented as vertices on the binary tree $\mathcal T_h$. Thus
$E_{\bar{\alpha}}$ is nonempty if and only if there is a strict lineage among the indices in
$\bar{\alpha}$, i.e., for any two entries of $\bar{\alpha}$, one is either an ancestor or a descendant of the other. In other
words, the vertices of $\bar{\alpha}$ lie on a ray of $\mathcal T_h$, i.e. $\bar{\alpha}
\in \mathbb M_{h}'$. It follows from the definition \eqref{def-esets} and more precisely from Lemma \ref{Lemma 4}\eqref{EN-tree-system} that the sets $E_N$ shrink as $N$ proceeds down a ray of the tree. Thus $E_{\bar{\alpha}}$ must equal $E_{N_0}$, where $N_0 = ||\bar{\alpha}||_{\infty}$ is the terminating vertex of the ray that contains the indices of $\bar{\alpha}$. This leads to \eqref{set-vector-number}.      
\vskip0.1in
\noindent Part \eqref{F-G-G'} is obtained by adding the estimates
deduced in the first two parts of the lemma over all $\bar{\alpha} \in
\mathbb M_h$. The verification is left to the interested reader. Note the
importance of the large constant $C$ in this step, as a result of
which the error implicit in $\cong$ remains small even after summing over $\#(\mathbb M_h) = O(2^{2hp}) = O(2^{2mp_0})$
terms.  
\end{proof}

\subsection{Proof of Lemma \ref{Lemma 8}} 
\begin{proof}
In view of Lemma \ref{Lemma 5}\eqref{F-G-G'}, we know that $\mathtt G_h \cong
\mathtt G_h'$. To establish the first relation in \eqref{Gh-estimate},
it therefore suffices to show that $\mathtt G_h' \cong \mathtt
G_h^{\ast}$. This in turn will follow from the estimate below. For any choice of $h, p, q, r, s, \bar{\mu}, \bar{\nu}$, 
\begin{equation} \label{Galpha0}
\begin{aligned} \int G_{\bar{\alpha}}(x) \, dx &\cong 0 \text{ for }
\bar{\alpha}=(\bar{m}, \bar{n}) \in
\mathbb M_h'[r, \bar{\mu}; s, \bar{\nu}]\setminus \mathbb M_h^{\ast}[r, \bar{\mu}; s, \bar{\nu}], \text{ i.e., } \\   
\int G_{\bar{\alpha}}(x) \, dx &\cong 0 \text{ unless } \mu_r = \nu_s \text{
  and } m_r = n_s.    
\end{aligned} 
\end{equation}
By Lemma \ref{Lemma 5}\eqref{FsameG} combined with Plancherel's
theorem, we obtain 
\begin{eqnarray*}
\int G_{\bar{\alpha}}(x) \, dx \cong \int_{\RR^{n}} F_{\bar{\alpha}}(x) dx   \,\,  =  \,\,
\int_{\RR^{n}}  \widehat{f}^{(\mu_{1})}_{m_{1}} \ast \cdots \ast
\widehat{f}^{(\mu_{r})}_{m_{r}} (\xi) \times \overline{\widehat{f}^{(\nu_{1})}_{n_{1}}  \ast \cdots \ast\widehat{f}^{(\nu_{s})}_{n_{s}}}(\xi) d\xi,
\end{eqnarray*}
where we use the notation $h^{(l)}$ to denote the $l$-fold convolution
of $h$ with itself. 
The integrand on the right hand side above is supported in the set
\begin{eqnarray*}
\big( \mu_{1}Q_{m_{1}} + \ldots + \mu_{r}Q_{m_{r}} \big) \cap \big( \nu_{1}Q_{n_{1}} + \ldots + \nu_{s}Q_{n_{s}}  \big).
\end{eqnarray*}
By \eqref{cubespread} in Lemma \ref{Lemma
  3}\eqref{intersection}, this intersection is empty unless $m_{r} =
n_{s}$ and $\mu_{r} = \nu_{s}$, establishing \eqref{Galpha0}. 
\vskip0.1in
\noindent For the second relation in \eqref{Gh-estimate}, we will rely
on the following recursion formula, to be proven shortly:
\begin{equation} \label{recursion} 
\begin{aligned} \mathtt G_h'[r, \bar{\mu}; s, \bar{\nu}] &\cong \sum_{h_1=1}^{h-1} \mathtt
G_{h_1}'[r-1, \bar{\mu}^{[1]};s-1, \bar{\nu}^{[1]}], \text{ where } \\  \bar{\mu}^{[1]} &= (\mu_1,
\cdots, \mu_{r-1}), \; \bar{\nu}^{[1]} = (\nu_1, \cdots, \nu_{s-1}) 
\end{aligned} 
\end{equation}      
are multiplicity vectors of lengths $r-1$ and $s-1$ for the exponents $p - \mu_r$ and $q - \nu_s$ respectively. 
Assuming this for now, the proof is completed as follows. 
\vskip0.1in
\noindent First suppose that $\bar{\mu} \ne \bar{\nu}$, and that $0 \leq t \leq \min(r, s)$
is the smallest index such that $\mu_{r-t} \ne \nu_{s-t}$. If $t = 0$,
then $\mathtt G_h \cong 0$ directly from \eqref{Galpha0}. If $t > 0$,
then a $t$-fold iteration of \eqref{recursion} yields 
\begin{align} 
\mathtt G_h[r, \bar{\mu}; s, \bar{\nu}] \cong \mathtt G_h'[r, \bar{\mu}; s, \bar{\nu}] & \cong \sum_{h_1=1}^{h-1}
\mathtt G_{h_1}'[r-1, \bar{\mu}^{[1]};s-1, \bar{\nu}^{[1]}] \cong \cdots  \nonumber \\
&\cong \sum_{h_1=1}^{h-1} \cdots \sum_{h_t = 1}^{h_{t-1}-1} \mathtt G_{h_t}'[r-t, \bar{\mu}^{[t]};s-t, \bar{\nu}^{[t]}],  \label{last-quantity}
\end{align} 
with $\bar{\mu}^{[t]} = (\mu_1, \cdots, \mu_{r-t})$, $\bar{\nu}^{[s]} =
(\nu_1, \cdots, \nu_{s-t})$. Note that $\mu^{[t]}$ and $\nu^{[t]}$ are multiplicity vectors of length $r-t$ and $s-t$ for the exponents $p- \rho$ and $q - \rho$ respectively, where $\rho = \mu_{r-t+1} + \cdots + \mu_r = \nu_{s-t+1} + \cdots + \nu_s$. Since $\mu_{r-t} \ne \nu_{s-t}$, we can apply \eqref{Galpha0}, with the parameters $h, p, q, r, s, \bar{\mu}, \bar{\nu}$ in \eqref{Galpha0}
replaced by $h_t, p - \rho, q - \rho, r-t, s-t, \bar{\mu}^{[t]}, \bar{\nu}^{[t]}$ respectively. This leads to the estimate  
\[ \int G_{\bar{\alpha}}(x) \, dx  \cong 0 \text{ for every } \bar{\alpha} \in \mathbb M_{h_t}'[r-t, \bar{\mu}^{[t]}; s-t, \bar{\nu}^{[t]}].\]
After summing over all the indices $\bar{\alpha}$ in the relevant collection $\mathbb M_h'$, the above relation yields that $\mathtt G_{h_t}' [r-t, \bar{\mu}^{[t]}; s-t, \bar{\nu}^{[t]}] \cong 0$. This in turn shows that the iterated sum in \eqref{last-quantity}
is also $\cong 0$, since the number of summands is at most $h^t = O(m^p)$. This completes the proof for $\bar{\mu} \ne \bar{\nu}$. 
\vskip0.1in
\noindent On the other hand, if $\bar{\mu} = \bar{\nu}$, then
iterating \eqref{recursion} $r=s$ times we find that 
\begin{align*} \mathtt G_h[r, \bar{\mu}; s, \bar{\mu}] \cong\mathtt G_h'[r, \bar{\mu}; s, \bar{\mu}] &\cong \sum_{h_1=1}^{h-1}
\cdots \sum_{h_{r-1}=1}^{h_{r-2}-1} \mathtt G_{h_{r-1}}'[1, \mu_1; 1,
\mu_1] \\ &= \sum_{h_1=1}^{h-1}
\cdots \sum_{h_{r-1}=1}^{h_{r-2}-1} 1 = O(h^{r-1}).  \end{align*}
At the penultimate step above, we have computed for any $h_{r-1} = l$,    
\[ \mathtt G_{l}'[1, \mu_1; 1, \mu_1] = \sum_{m_1: h(m_1) = l}\int_{\mathbb
R^n} \chi_{m_1} =  \int_Q \sum_{m_1: h(m_1) = l} \chi_{m_1} = 1.  \] 
The last step is a consequence of Lemma \ref{Lemma 4}\eqref{sum-over-ht} with $N_0=1$. This completes the proof.  
\end{proof}

\subsection{Summing over subtrees: Proof of \eqref{recursion}} 
\begin{proof}
Any $\bar{\alpha} = (\bar{m}, \bar{n}) \in \mathbb M_{h}^{\ast}[r,
\bar{\mu}; s, \bar{\nu}]$ can be written as $\bar{\alpha} = (\bar{m}',
m_r; \bar{n}', m_r)$, where \[h(m_r) = h, \quad \bar{m}' = (m_1,
\cdots, m_{r-1}), \quad \bar{n}' = (n_1, \cdots, n_{s-1}), \]  and $\bar{\alpha}' = (\bar{m}', \bar{n}')$ is
a string of vertices lying on a ray in $\mathcal T_m$ and terminating
in the vertex $||\bar{\alpha}'||_{\infty} = \max(m_{r-1}, n_{s-1})$ of height $<h$. As such, $\mathbb
M_{h}^{\ast}$ can be partitioned as 
\[ \mathbb M_h^{\ast}[r, \bar{\mu}; s, \bar{\nu}] =
\bigsqcup_{h_1=1}^{h-1} \left\{ \bar{\alpha} = (\bar{m}', m_r, \bar{n}',
m_r) \Bigl| \;\begin{aligned}  &\bar{\alpha}' = (\bar{m}', \bar{n}') \in \mathbb M_{h_1}'[r-1,
\bar{\mu}'; s-1, \bar{\nu}']  \\ &h(m_r) = h \end{aligned} \right\}. \] 
This results in a corresponding decomposition for the sum representing
$\mathtt G_h'$: 
\begin{equation} \label{three sums} \begin{aligned}
\mathtt G_h'[r, \bar{\mu}; s, \bar{\nu}] &\cong \mathtt G_h^{\ast}[r,
\bar{\mu}; s, \bar{\nu}] \\ 
&= \sum_{h_1=1}^{h-1} \sum_{\bar{\alpha}', m_r}\int
G_{\bar{\alpha}'}(x) \chi_{m_r}(x) \, dx \\ 
&= \sum_{h_1=1}^{h-1} \sum_{\bar{\alpha}'} \int
G_{\bar{\alpha}'}(x) \sum_{m_r} \chi_{m_r}(x) \, dx \\
&=\sum_{h_1=1}^{h-1} \sum_{\bar{\alpha}'} \int G_{\bar{\alpha}'}(x) \chi_{E_{\bar{\alpha}'}}(x) \,
dx \\
&= \sum_{h_1=1}^{h-1} \sum_{\bar{\alpha}'} \int G_{\bar{\alpha}'}(x)
\, dx = \sum_{h_1=1}^{h-1} \mathtt G_{h_1}'[r-1, \bar{\mu}^{[1]}; s-1, \bar{\nu}^{[1]}].  
\end{aligned} \end{equation}  
In all the sums above, $\bar{\alpha}'$ ranges over $\mathbb
M_{h_1}'[r-1, \bar{\mu}'; s-1, \bar{\nu}']$. For a given
$\bar{\alpha}$, the summation index $m_r$
ranges over descendants of $||\bar{\alpha}'||_{\infty}$ of height $h$
in $\mathcal T_m$. This has been described in Figure \ref{fig6}. 
In the third equality, the summation in $m_r$ follows from the property that $\{\chi_N\}$ is a tree system. In particular we have
invoked Lemma \ref{Lemma 4}\eqref{sum-over-ht} with
$N_0=||\bar{\alpha}'||_{\infty}$, along with \eqref{set-vector-number}.   
\end{proof} 

\begin{figure}[htbp]\centering\includegraphics[width=4.in]{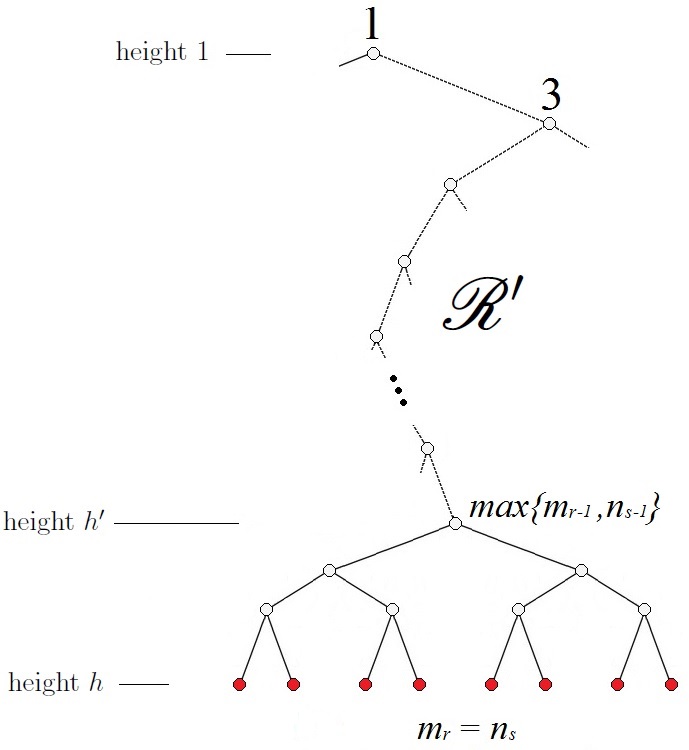}
\caption{The process of summation in \eqref{three sums}: the fixed vertices $m_{1}, \ldots m_{r-1}, n_{1}, \ldots n_{s-1}$ lie on the ray $\mathscr{R}'$, from the root of $\mathcal{T}_{m}$ to the vertex $\max\left\{m_{r-1}, n_{s-1}\right\}$. The innermost summation (in $m_r$) is over all vertices at height $h$ of the subtree rooted at $\max(m_{r-1}, n_{s-1})$.} 
\label{fig6}
\end{figure}



\section{Proof of Lemma \ref{Lemma 9}} \label{geometric-section} 

\noindent In this section we prove the geometric result used in the proof of the theorem in Section \ref{proof of thm}.
Recall that $\Gamma_{v} := \left\{x \in \RR^{n} : x \cdot v > 0\right\}$. We restate the lemma below for easier referencing.

\begin{lemma} 
Let $U$ be a set of unit vectors in $\RR^n$, all pointing in distinct directions. 
Assume that $\# U = M$ for some $M\geq 2$, and that
all vectors $v\in U$ obey $v\cdot e_n>0$, where $e_n=\left(0, \ldots 0, 1\right)$.
Then there is an ordering $\{u_1,\dots,u_{M}\}$ of vectors in $U$, and a collection of
pairwise disjoint sectors $S_{1}, \ldots S_{M-1}\subset \RR^{n}$ (see Definition \ref{Definition of sector}), such that, up to sets of
Lebesgue measure 0, we have for $l=2,\dots,M$
\begin{equation}\label{sector-inclusion-b}
\Gamma_{u_l} \cap S_i =
 \begin{cases} 
S_i &\text{ if } i < l, \\ \emptyset &\text{ if } i \geq l.
\end{cases} 
\end{equation}
\end{lemma}

\begin{proof} 
For a unit vector $v\in \RR^{n}$, we will use $\pi_{v}$ to denote the hyperplane 
$\left\{x \in \mathbb{R}^{n} : x \cdot v = 0 \right\}$. By a slight abuse of notation, we will also write 
$\Gamma^{c}_{v}=\left\{x \in \RR^{n} : x \cdot v < 0\right\}$.
\vskip0.1in
\noindent For $x' \in \RR^{n-1}$, let $r_{x'}$ be the line
\begin{eqnarray*}
r_{x'} =  \left\{(x', t) \in \mathbb{R}^{n} :\  t \in \mathbb{R}\right\}.
\end{eqnarray*}
Since $v\cdot e_n>0$ for all $v\in U$, the line $r_{x'}$ is not parallel to any of the corresponding hyperplanes $\pi_v$.
Moreover, for all $x'$ outside of an exceptional set of $(n-1)$-dimensional Lebesgue measure 0, it intersects 
these hyperplanes at distinct points. Fix one such point $x'$, and let $A_i=(x',t_i)$ be the intersection points listed in
the order of decreasing $t$ so that $t_{1} > t_{2} > \cdots > t_{M}$. We then label the vectors in $U$ as $u_1,\dots, u_M$
so that
\begin{align*} 
A_{i} &\in \pi_{u_i},\ i=1,\dots,M,
\end{align*}
and define 
\begin{align*}
S_{i} :=  \Gamma_{u_{1}}^c \cap \ldots \cap \Gamma_{u_{i}}^c \cap \Gamma_{u_{i+1}} \cap \ldots \cap \Gamma_{u_{M}},\ 
i=1,\dots,M-1.
\end{align*} 
Then $S_i \subset \Gamma^c_{u_l}$ if $l\leq i$ and 
$S_i\subset \Gamma_{u_l}$ if $l > i$, so that we have (\ref{sector-inclusion-b}). It is also clear from the definition that the
sectors $S_i$ are pairwise disjoint.
\vskip0.1in
\noindent To see that they are non-empty, it suffices to check that $B_i\in S_i$ for $i=1,\dots,M-1$, where
$B_{i}=(x',\tau_{i})$ for some choice of scalars $\tau_i$ obeying $t_{i} > \tau_{i} > t_{i+1}$. Indeed, for any $1 \leq i  \leq M-1$
and  $1 \leq l  \leq M$, we have
\begin{align*} 
\vec{OB_i}\cdot u_l = &(\vec{OA_l}+\vec{A_lB_i})\cdot u_l = \vec{A_lB_i}\cdot u_l = (\tau_i-t_l)\,(e_n\cdot u_l), \\ \text{ which is } &\begin{cases} <0 &\text{ for } l \leq  i \text{ since } \tau_i < t_l, \\   >0 &\text{ for } l > i \text{ since } \tau_i > t_l. \end{cases} 
\end{align*}
Thus $B_i \in \Gamma_l^{c}$ if $ l \leq  i$ and $B_i\in \Gamma_l$ if $ l > i$, proving the claim.  
\end{proof}
\begin{figure}[htbp]\centering\includegraphics[width=3.in]{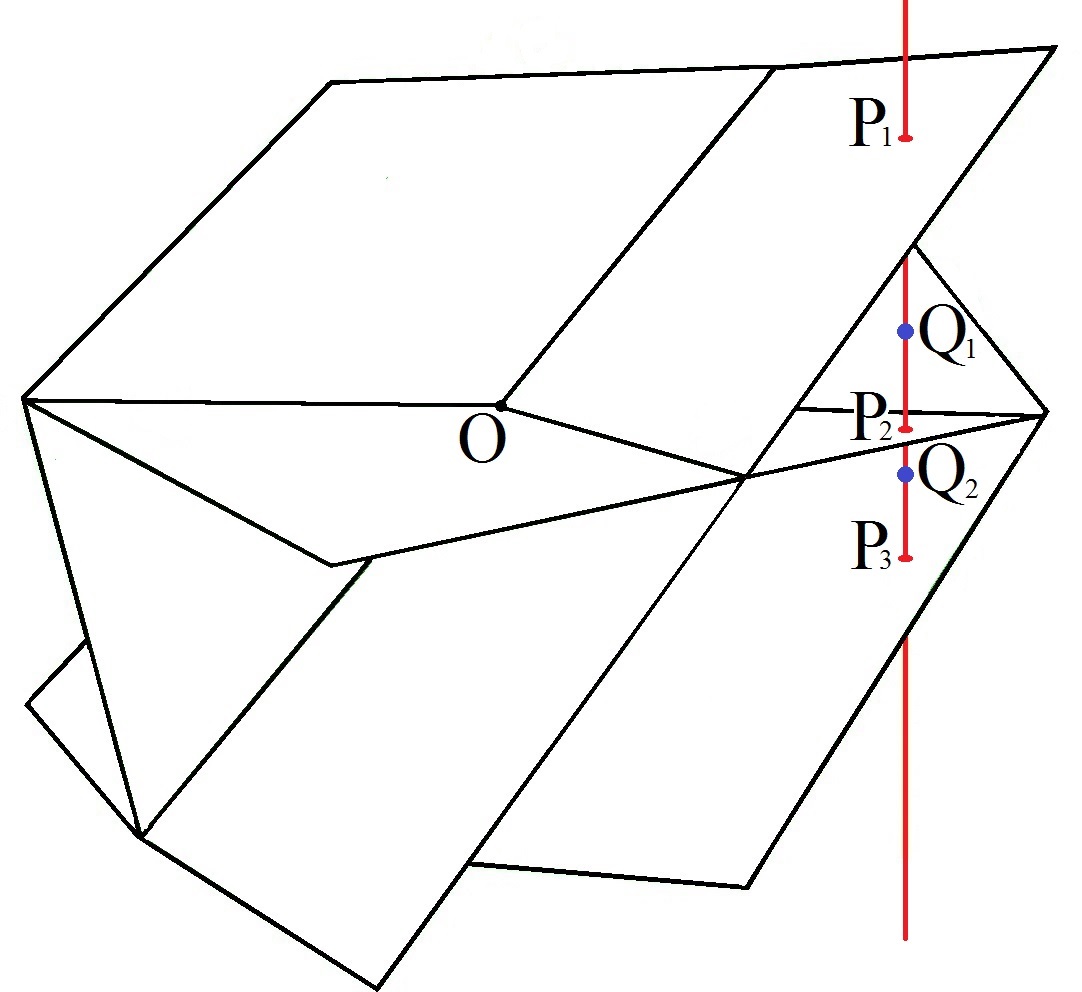}
\caption{An example of the geometric construction in Lemma \ref{Lemma 9} with $n=3$ and $M = 3$. The vertical line $r_{x'}$ intersects all three planes at distinct points.}
\label{fig7}
\end{figure}
\vskip0.1in
\noindent {\em{Remark: }} We point out the main distinctions of Lemma \ref{Lemma 9} in general dimensions relative to its planar counterpart in \cite{Karagulyan}. In dimension two, the hyperplanes $\pi_v$ are lines passing through the origin. Any conical sector is bounded by exactly two such lines. Thus $M$ lines of the form $\pi_v$ divide a half-plane into exactly $M+1$ sectors that admit an obvious ordering simply by moving in a clockwise direction. In $\mathbb R^n$, hyperplanes intersect in more complicated ways. A conical sector may be bounded by a number of hyperplanes far greater than $n$. Furthermore, a collection of $M$ vectors in $U$ typically generates many more than $M$ sectors, among which there is no natural ``global" ordering. In Lemma \ref{Lemma 9}, we choose, from the collection of all sectors, a subset of size $M-1$, on which we impose a natural ordering, in terms of the height of the sector above a fixed point in the $\{x_n = 0\}$-hyperplane.  

\bigskip

\section{Acknowledgement}
The authors thank two anonymous referees for their careful reading of the manuscript and valuable suggestions that improved the presentation. 
This work was supported by NSERC Discovery Grants 22R80520 and 22R82900.

\vskip0.2in
\noindent \author{\textsc{Izabella \L aba}}\\
University of British Columbia, Vancouver, Canada. \\
Electronic address: \texttt{ilaba@math.ubc.ca}
\vskip0.2in
\noindent \author{\textsc{Alessandro Marinelli}}\\
University of British Columbia, Vancouver, Canada. \\
Electronic address: \texttt{marine7@math.ubc.ca}
\vskip0.2in 
\noindent \author{\textsc{Malabika Pramanik}}\\
University of British Columbia, Vancouver, Canada. \\
Electronic address: \texttt{malabika@math.ubc.ca}

\end{document}